\documentclass[11pt]{article}

\usepackage{soul}

\usepackage[utf8]{inputenc}
\usepackage[T1]{fontenc}
\usepackage{textcomp}
\usepackage{amsmath,amssymb,amsthm}
\usepackage{lmodern}
\usepackage[a4paper, margin=0.8in]{geometry}
\usepackage{xcolor, pict2e}
\usepackage{microtype}
\usepackage{listings}
\usepackage{moreverb}
\usepackage{hyperref}
\usepackage[sans]{dsfont}
\usepackage[font={sf,small}, labelfont={sf,bf}, margin=1cm]{caption}
\usepackage{enumerate}

\usepackage{import}
\usepackage{xifthen}
\usepackage{pdfpages}
\usepackage{transparent}
\usepackage{chngcntr}
\counterwithin{figure}{section}
\usepackage{subcaption}

\definecolor{auburn}{rgb}{0.43, 0.21, 0.1}

\usepackage{hyperref}
\hypersetup{
    colorlinks=true,
    linkcolor=blue,
    citecolor=auburn,
    urlcolor=blue,
    pdfborder={0 0 0}
}

\newtheorem*{theorem*}{Theorem}
\newtheorem{theorem}{Theorem}[section]
\newtheorem{proposition}[theorem]{Proposition}
\newtheorem{lemma}[theorem]{Lemma}
\newtheorem{corollary}[theorem]{Corollary}
\newtheorem{notation}[theorem]{Notation}
\newtheorem{definition}[theorem]{Definition}

\newtheorem{assumption}[theorem]{Assumption}
\newtheorem{example}[theorem]{Example}
\theoremstyle{remark}
\newtheorem{remark}[theorem]{Remark}

\newcommand{\R}{\mathbb{R}}

\newcommand{\C}{\mathbb{C}}

\newcommand{\D}{\mathbb{D}}

\newcommand{\Ec}{\mathcal{E}}
\newcommand{\Fc}{\mathcal{F}}

\newcommand{\Kc}{\mathcal{K}}
\newcommand{\Lc}{\mathcal{L}}

\newcommand{\Pc}{\mathcal{P}}

\newcommand{\Prob}[1]{\mathbb{P} \left( #1 \right) }

\renewcommand{\P}{\mathbb{P}}
\newcommand{\E}{\mathbb{E}}

\newcommand{\floor}[1]{\left\lfloor #1 \right\rfloor}

\newcommand{\indic}[1]{ \mathbf{1}_{ \left\{ #1 \right\} } }
\newcommand{\eps}{\varepsilon}

\DeclareMathOperator{\dist}{dist}
\renewcommand{\d}{\mathrm{d}}

\newcommand{\loopm}{\mu^{\mathrm{loop}}}

\def \loopmeasure{\mu^{\rm loop}}
\def \loopm{\mu^{\rm loop}}
\newcommand{\bubm}[1]{\mu^{#1, \mathrm {bub}}}

\def \inte{{\rm int}}
\def \ext{{\rm ext}}
\def \out{{\rm out}}
\def \diam{{\rm diam}}
\def \area{{\rm Area}}
\newcommand{\Hb}{\mathbb{H}}
\newcommand{\Cb}{\mathbb{C}}
\newcommand{\Eb}{\mathbb{E}}
\newcommand{\Pb}{\mathbb{P}}
\newcommand{\Db}{\mathbb{D}}
\newcommand{\Rb}{\mathbb{R}}
\renewcommand{\Im}{\mathrm{Im}}
\renewcommand{\Re}{\mathrm{Re}}
\def \wired{{\rm wired}}
\def \SLE{{\rm SLE}}
\newcommand{\wt}{\widetilde}
\def \pint{\P^\wired_{\D, \partial \D}}
\def \Eint{\E^\wired_{\D, \partial \D}}
\newcommand{\corint}[1]{\langle#1\rangle^\wired_{\D, \partial \D}}

\newcommand{\ol}{\overline}

\title{The height gap of planar Brownian motion is $\frac{5}{\pi}$}
\author{Antoine Jego\thanks{CNRS \& CEREMADE, Université Paris-Dauphine, PSL University, France; antoine.jego@dauphine.psl.eu} \and Titus Lupu\thanks{Sorbonne Université and Université Paris Cité, CNRS, Laboratoire de Probabilités, Statistique et Modélisation, F-75005 Paris, France; titus.lupu@sorbonne-universite.fr} \and Wei Qian\thanks{Department of Mathematics, Run Run Shaw Building, The University of Hong Kong, Pokfulam, Hong Kong; on leave from CNRS, Laboratoire de Math\'ematiques d'Orsay, Universit\'e Paris-Saclay; weiqian0@hku.hk}}
\date {}

\numberwithin{equation}{section}

\begin{document}

\maketitle

\abstract{
We show that the occupation measure of planar Brownian motion exhibits a constant height gap of $5/\pi$ across its outer boundary.
This property bears similarities with the celebrated results of Schramm--Sheffield \cite{MR3101840} and Miller--Sheffield \cite{MS} concerning the height gap of the Gaussian free field across SLE$_4$/CLE$_4$ curves. Heuristically, our result can also be thought of as the $\theta \to 0^+$ limit of the height gap property of a field built out of a Brownian loop soup with subcritical intensity $\theta>0$, proved in our recent paper \cite{JLQ23b}.
To obtain the explicit value of the height gap, we rely on the computation by Garban and Trujillo Ferreras \cite{MR2217292} of the expected area of the domain delimited by the outer boundary of a Brownian bridge.}

\setcounter{tocdepth}{1}
\tableofcontents

\section{Introduction}

In the eighties, Mandelbrot \cite{MR0665254} conjectures that the outer boundary of planar Brownian motion is a fractal curve whose Hausdorff dimension equals $4/3$.
The outer boundary, and more generally the geometry of planar Brownian motion, has attracted a lot of attention ever since (many results and references can be found in the book \cite{morters2010brownian}).
We mention that, in \cite{MR1879851,MR1961197} (also see \cite{Lawler2000TheDO}), Lawler, Schramm and Werner prove Mandelbrot's conjecture. Actually, the outer boundary turns out to be a random continuous self-avoiding curve which is distributed as a version of Schramm--Loewner Evolution (SLE) with parameter $\kappa = 8/3$ \cite{MR1992830}.
The main result of the current article is that occupation measure of a 2D Brownian trajectory exhibits a constant height gap of $5/\pi$ across its outer boundary. Due to conformal invariance of planar Brownian motion (in particular invariance under the inversion map $z\mapsto 1/z$), the same result in fact also holds for the boundary of any connected component of the complement of a 2D Brownian trajectory.

\medskip

By definition, the outer boundary of a planar Brownian motion $\Pc$, that we denote by $\out(\Pc)$, is the boundary of the unbounded connected component of $\C \setminus \Pc([0,T(\Pc)]),$ where $T(\Pc)$ stands for the lifetime, or duration, of $\Pc$.
On the other hand, the occupation measure of $\Pc$, that we denote by $L_x(\Pc) \d x$, is a random Borel measure on $\C$ that assigns to each Borel set $A \subset \C$ the total time spent by $\Pc$ in $A$:
\begin{equation}
\label{E:occ_measure}
\int_A L_x(\Pc) \d x = \int_0^{T(\Pc)} \indic{\Pc_t \in A} \d t.
\end{equation}
Although we write $L_x(\Pc)\d x$, we emphasise that the occupation measure is almost surely not absolutely continuous with respect to Lebesgue measure.

To state our main result, let $\Pc = (\Pc_t)_{t \in [0,1]}$ be a standard two-dimensional Brownian motion which starts at the origin and denote by $\nu$ the law of its outer boundary $\out(\Pc)$. 
For any continuous loop $\gamma$, we denote by $\inte(\gamma)$ the complement of the closure of the unbounded component of $\C\setminus\gamma$.
The measure $\nu$ gives a positive mass to simple loops. But it also gives a positive mass to loops $\gamma$ with cut points, in which case $\inte(\gamma)$ is formed of countably infinitely many connected components, called beads in \cite{zbMATH02055258}.

\begin{theorem}\label{T:intro}
For $\nu$-almost all loop $\gamma$, the following holds.
Let $\xi \subset \gamma$ be a simple curve (or $\xi$ can be $\gamma$ if $\gamma$ is a simple loop), which is measurable with respect to $\gamma$.
Let $(f_\eps)_\eps$ be a sequence of test functions $f_\eps : \inte(\gamma) \to \R$ that are measurable with respect to $\gamma$ and such that $\int f_\eps = 1$ and $\{ f_\eps \neq 0 \} \subset \{ x \in \inte(\gamma), \d(x,\xi) < \eps \}$ for all $\eps >0$. Assume further that they satisfy the integrability conditions of Assumption~\ref{assumption_feps_main}.
Then
\begin{equation}
\label{E:T_intro_bc}
\int f_\eps(x) L_x(\Pc) \d x \xrightarrow[\eps \to 0]{} \frac{5}{\pi} \qquad \text{in} \quad \mathrm{L}^1(\P(\cdot \vert \out(\Pc) = \gamma)).
\end{equation}
\end{theorem}

Assumption~\ref{assumption_feps_main}, which will be stated later, holds for a wide range of functions. In particular, we can consider a sequence $(f_\eps)_\eps$ concentrated near a tiny given portion of the boundary, not the entire boundary. The property near such a tiny portion does not depend on the behavior of the Brownian motion far away.
Therefore, Theorem~\ref{T:intro} can be read as an almost-sure statement about the local behavior of a planar Brownian motion. The global behavior of the Brownian motion (such as time duration, endpoints) is not important. In Theorem \ref{T:main}, we prove an analogous result for a Brownian loop.

Note that Theorem~\ref{T:intro} is stated directly for the outer boundary $\gamma$ which is an SLE$_{8/3}$-type fractal curve. One can in fact also map such a fractal curve to a smooth curve via conformal maps. In Theorem~\ref{T:level_line}, we conformally send the domain encircled by the outer boundary of a Brownian loop onto the unit disk, and show that the resulting occupation field in the unit disk also has the same height gap $5/\pi$ on the boundary. 


\paragraph{Height gap of the Gaussian free field and Brownian loop soup}

Theorem \ref{T:intro} is a height gap property for the occupation measure across its outer boundary $\gamma$: it is constant and equal to $5/\pi$ on $\gamma$ when approaching $\gamma$ from inside, whereas it is equal to zero on $\gamma$ when approaching $\gamma$ from outside.

As discovered by Schramm--Sheffield \cite{MR3101840} and Miller--Sheffield \cite{MS}, the Gaussian free field also exhibits such height gap properties. Indeed, one can couple it with a collection of self-avoiding loops, the conformal loop ensemble CLE$_4$ with parameter $\kappa=4$. Outside of these loops (in the CLE carpet), the free field vanishes, whereas inside each loop the value of the field jumps by an additive factor $\pm 2 \lambda$, where $\lambda >0$ is explicit.


In \cite{JLQ23b}, we construct and study the properties of a conformally invariant field $h_\theta$ defined out of a Brownian loop soup with parameter $\theta \in (0,1/2]$ (a Poisson point process of Brownian loops with intensity $2\theta$ times a loop measure $\loopm$ \eqref{E:loopm}). Informally, $h_\theta$ corresponds to a signed version of the local time of the loop soup to the power $1-\theta$. When $\theta =1/2$, this field coincides with a Gaussian free field. 
We show in \cite{JLQ23b} that, surprisingly, the height gap property generalises to any value of $\theta \in (0,1/2]$ across self-avoiding loops distributed according to CLE$_\kappa$, where $\kappa = \kappa(\theta) \in (8/3,4]$. The value of the height gap is however explicitly known only when $\theta=1/2$.
Theorem \ref{T:intro} can be thought of as the limit $\theta \to 0^+$ of this result. Informally, $h_\theta$ becomes simply the local time of a single loop and the CLE$_\kappa$ becomes an SLE$_{8/3}$ loop (recall that $\kappa(\theta) \to 8/3$ as $\theta \to 0^+$).

We remark that for a critical loop soup, the excursions induced by the loops that touch the cluster boundary is a Poisson point process of excursions \cite{QianWerner19Clusters}, hence the occupation field of these excursion on the boundary is constant by a law of large numbers. This constant can be shown to be $\pi/4$, thanks to isomorphism with the GFF. In our case, the set of excursions induced by a Brownian loop (by Proposition~\ref{prop:pp}) is not a Poisson point process (by the arguments of \cite[Section 3]{MR3901648}). It has a priori a rather intricate law, but interestingly it still has a constant boundary occupation time $5/\pi$. We expect that for a loop-soup with intensity $\theta$ going from $0$ to $1/2$, the occupation field induced by the loops that touch the cluster boundary should vary continuously from $5/\pi$ to $\pi/4$. It is an interesting open question to work out the exact value as a function of $\theta$.

Let us however emphasise that the current paper does not rely on the height gap property for the loop soup. Indeed, trying to make sense of the limit $\theta \to 0^+$ of the loop soup result would be difficult; for instance, two loop soups with distinct intensities $\theta_1$ and $\theta_2$, even when naturally coupled, are mutually singular and their respective clusters are drastically different.
Instead, we will work directly with a Brownian trajectory. In addition and contrary to the loop soup case, we compute explicitly the value of the height gap in our setting.

\paragraph{Some proof ideas}

We now comment on some of the ideas involved in the proof of Theorem \ref{T:intro}. We start by considering a Brownian trajectory which possesses a higher level of symmetry than a plain Brownian motion stopped after one unit of time. The trajectory that we will consider will be sampled according to the Brownian loop measure. Proceeding as in \cite[Section 3]{MR3901648}, we conformally map the inside of the outer boundary of the loop to the unit disc $\D$, and get a Brownian loop in $\D$, ``conditioned'' on the event that its outer boundary agrees with the unit circle $\partial \D$. We will denote the resulting law by $\pint$. It has the nice feature that it is invariant under any conformal map preserving $\D$. As a consequence of this symmetry, we will show in Lemma~\ref{L:1point_constant} that the expectation of the associated occupation measure is a constant $\lambda_0$ times Lebesgue measure in $\D$. 

The computation of the value of the constant $\lambda_0$ is achieved in Lemma \ref{L:lambda}. As we will see, we relate it to the expectation of the area of the domain delimited by $\gamma$ (it naturally appears for instance on the right hand side of \eqref{E:T_main_expectation} when one takes $f=1$). We eventually rely on a result of Garban and Trujillo Ferreras \cite{MR2217292}: the expected area of the domain delimited by the outer boundary of a Brownian bridge of duration 1 equals $\pi/5$.

At this stage, we know that, under the law $\pint$, the occupation measure has a constant expectation which has an explicit value. We want to show that it concentrates around its expectation when integrated against test functions whose supports are close to the unit circle.
To do so, we will show in Lemma \ref{L:second_good} that the occupation measures of two small sets, near $\partial \D$ and at a macroscopic distance to each other, become asymptotically uncorrelated. This decorrelation property consists in the most involved part of the paper.
It starts with a partial exploration of the outer boundary of a Brownian loop together with a decomposition of the resulting loop into conditionally independent excursions (analogous results were obtained for loop soups in \cite{QianWerner19Clusters,MR3901648,qian2018}). As a key tool, we will also use estimates that we derived in \cite{JLQ23b} for general point processes of Brownian excursions which satisfy some conformal restriction property; see Corollaries \ref{C:JLQ}, \ref{C:interval_excursion} and \ref{C:upper_bound_2point}.

\medskip

We finish this introduction by stating the integrability assumptions satisfied by the sequence of functions $(f_\eps)_\eps$ in Theorem \ref{T:intro}. Given a continuous loop $\gamma$, we assume:

\begin{assumption}\label{assumption_feps_main}
For all $\eps >0$,
$\int_{\inte(\gamma)} f_\eps = 1$ and $\{ f_\eps \neq 0 \} \subset \{ x \in \C, \d(x,\gamma) < \eps \}$.
Moreover,
\begin{gather}
\label{E:assumption_feps_main}
\sup_\eps \int_{\inte(\gamma) \times \inte(\gamma)} \max(1, -\log|x-y|) f_\eps(x) f_\eps(y) \,\d x \,\d y < \infty,\\
\label{E:assumption_feps_main2}
\text{and} \quad \lim_{\delta \to 0} \limsup_{\eps \to 0} \int_{\inte(\gamma) \times \inte(\gamma)} \indic{|x-y| < \delta} \max(1, -\log|x-y|) f_\eps(x) f_\eps(y) \,\d x \,\d y = 0.
\end{gather}
In Section \ref{S:proof_main}, we give a broad family of examples of test functions satisfying these assumptions; see Example \ref{Example} and Lemma \ref{L:equivalence_assumptions}.
\end{assumption}

\paragraph{Organisation of the paper}

\begin{itemize}
    \item Section \ref{S:preliminaries}: We define more precisely our setup and state a version of Theorem \ref{T:intro} for a path distributed according to the Brownian loop measure (see Theorem \ref{T:main}). We will then recall and prove some preliminary results.
    \item Section \ref{S:conditioning}: We show that we can condition a Brownian loop on a portion of its outer boundary and study the resulting probability law.
    \item Section \ref{S:proof_main}: This section is the main contribution of the current paper where we prove our main result Theorem \ref{T:intro}.
\end{itemize}

\section{Setup and preliminaries}\label{S:preliminaries}

In this section, we will first describe precisely the setup we will be working with. We will then state a version of Theorem \ref{T:intro} that is of independent interest for a path ``sampled'' according to the loop measure $\loopm$ \eqref{E:loopm}; see Theorem \ref{T:main}. The rest of the section will then be dedicated to some preliminary results.

\medskip

\noindent\textbf{The Brownian loop measure.}
Our approach to Theorem \ref{T:intro} crucially relies on the Brownian loop measure. 
Let us recall its definition and some important facts.
As introduced by \cite{MR2045953},
the Brownian loop measure on the plane $\C$ is defined by
\begin{equation}\label{E:loopm}
\loopm(\d \Pc) = \int_\C \d z \int_0^\infty \frac{\d t}{t} \frac{1}{2\pi t} \P^{t,z,z}(\d \Pc),
\end{equation}
where $\P^{t,z,z}$ denotes the probability law of a Brownian bridge from $z$ to $z$ with duration $t$. In this paper, we consider standard Brownian motion with generator $\frac12 \Delta$. We will view the measure $\loopm$ as a measure on the following space $\Lc$ of loops:

\begin{definition}[Space of loops $\Lc$]\label{Def:space_loops}
It is the space of unrooted loops, up to monotone time reparametrisation. It is defined as the space of continuous maps from the unit circle $\mathbb{S}^1$ to $\C$ where two maps $\Pc_1$ and $\Pc_2$ are identified if there exist $\theta_0 \in [0,2\pi]$ and an increasing bijection $\sigma : [0,2\pi] \to [0,2\pi]$ such that $\Pc_1(e^{i(\theta+\theta_0)}) = \Pc_2(e^{i \sigma(\theta)})$ for all $\theta \in [0,2\pi]$. We will endow $\Lc$ with the following metric: for two loops $\Pc_1$ and $\Pc_2$, let
\begin{equation}
\label{E:distance_Lc}
\d_\Lc(\Pc_1,\Pc_2) = \inf_{\theta_0,\sigma} ~ \sup ~\{ | \Pc_1(e^{i(\theta+\theta_0)}) - \Pc_2(e^{i \sigma(\theta)}) |, \theta \in [0,2\pi] \}
\end{equation}
where the infimum runs over $\theta_0 \in [0,2\pi]$ and increasing bijections $\sigma : [0,2\pi] \to [0,2\pi]$. We will then use the topology and $\sigma$-algebra naturally associated to this metric.
\end{definition}

\begin{definition}[Space of self-avoiding loops $\Gamma$]\label{Def:space_self-avoiding}
It is the subset of $\Lc$ consisting of injective maps $\mathbb{S}^1$ to $\C$. We denote this space by $\Gamma$. Following \cite{MR2350053}, it is more natural to consider the following $\sigma$-algebra. An annular region $A \subset \C$ is a subset of $\C$ that is conformally equivalent to an annulus $\{z \in \C: 1< |z| < R \}$. For any such $A$, let $\Gamma_A$ be the subset of self-avoiding loops which are included in $A$ and that separates the two boundary components of $A$. We will equip $\Gamma$ with the $\sigma$-algebra generated by this family.
\end{definition}

It will be convenient for us to consider loops as unrooted and up to time parametrisation since, in this way, the loop measure $\loopm$ is conformally invariant \cite{MR2045953}. We now mention a few more important results.
The push forward of the measure $\loopm$ by the map $\Pc \in \Lc \mapsto \out(\Pc) \in \Gamma$ (whose measurability can be checked by considering $\out^{-1}(\Gamma_A)$ for all annular region $A$ \cite{MR2350053}) is equal to an SLE$_{8/3}$ loop measure \cite{MR2350053}. This latter measure is a measure on the space $\Gamma$ that we denote by $\nu_\SLE$.
Moreover, the Brownian loop measure $\loopm$ factorises into
\begin{equation}
\label{E:decompo_loopmeasure}
\loopm(\d \Pc) = \int \loopm(\d \Pc \vert \gamma) \nu_\SLE(\d \gamma),
\end{equation}
where $\loopm(\d \Pc \vert \gamma)$ is now a probability measure: the Brownian loop measure ``conditioned on the event that the outer boundary agrees with $\gamma$''. Additionally and as shown in \cite{MR3901648}, if one conformally maps the interior of $\gamma$ to the unit disc $\D$, the push forward of $\loopm(\d \Pc \vert \gamma)$ is now a probability measure on loops in $\overline{\D}$ that does not depend on $\gamma$, or on the specific choice of conformal map.
See Proposition \ref{prop:decomp} below for a precise statement.

\medskip

\noindent \textbf{Recovering the time parametrisation.}
Although we consider loops, and more generally paths, up to time reparametrisation, we now explain that we can recover the original time parametrisation deterministically from the path modulo parametrisation.

To recover the original time parametrisation, one simply needs to notice that the occupation measure of a Brownian path $\Pc$ is solely a function of its trace $\Kc = \Pc([0,T(\Pc)])$.
Indeed, it is equal to the Minkowski content of $\Kc$ with gauge function $r \mapsto \frac1\pi r^2 |\log r|^2$. That is, for any open set $A \subset \C$,
\begin{equation}
    \label{E:Minkowski123}
\int_A L_x(\Pc) \d x = \frac1\pi \lim_{r \to 0} |\log r| \, |A \cap \Kc^r|
\end{equation}
where $\Kc^r$ is the $r$-enlargement of $\Kc$, i.e. the set of points at distance at most $r$ to $\Kc$, and $|A \cap \Kc^r|$ denotes the Lebesgue measure of $A \cap \Kc^r$.
The above convergence holds in $L^p$ for all $p \ge 1$.
This result is folklore and can be proven by the method of moments.
A related and more difficult result by Taylor \cite{Taylor_1964} states that the occupation measure agrees with a constant multiple of the Hausdorff measure of the path in the gauge $r \mapsto r^2 \log \frac1r \log \log \log \frac1r$.
As a result, the occupation measure, and thus the parametrisation of the path, can be recovered only from the path modulo time reparametrisation.
From now on, we will not mention time parametrisation any more and consider loops up to time reparametrisation, i.e. work with the space $\Lc$.

\medskip

We now state a version of Theorem \ref{T:intro} for $\Pc$ sampled according to $\loopm(\cdot \vert \gamma)$. It contains an additional property (\eqref{E:T_main_expectation} below) that is reminiscent of the conformal invariance of $\loopm$.

\begin{theorem}\label{T:main}
For $\nu_\SLE$-almost all $\gamma$, and $\Pc$ sampled according to $\loopm(\cdot \vert \gamma)$, the following holds.
\begin{itemize}
\item
(Constant expectation)
For all test function $f : \C \to \R$ measurable w.r.t. $\gamma$,
\begin{equation}
\label{E:T_main_expectation}
\loopm \Big[ \int f(x) L_x(\Pc) \d x \vert \gamma\Big] = \frac{5}{\pi} \int_{\inte(\gamma)} f.
\end{equation}
\item
(Constant boundary conditions) Let $(f_\eps)_\eps$ be a sequence of test functions $f_\eps : \inte(\gamma) \to \R$ that are measurable w.r.t. $\gamma$ and such that $\int f_\eps = 1$ and $\{ f_\eps \neq 0 \} \subset \{ x \in \inte(\gamma), \d(x,\gamma) < \eps \}$ for all $\eps >0$. Assume further that they satisfy the integrability conditions of Assumption~\ref{assumption_feps_main}.
Then
\begin{equation}
\label{E:T_main_bc}
\int f_\eps(x) L_x(\Pc) \d x \xrightarrow[\eps \to 0]{L^1} \frac{5}{\pi}.
\end{equation}
\end{itemize}
\end{theorem}

In the remaining of this section, we collect preliminary results that we will need in the paper.

\subsection{Green's function and Poisson kernel}

We recall the definition and some explicit expressions of the Green's function and Poisson kernel for ease of future reference.
For any domain $D$, let $G_D$ be the Green's function in $D$. In our normalisation,
\begin{equation}
\label{E:Green}
\forall x,y \in \D,~
G_\D(x,y) = \frac{1}{\pi} \log \frac{|1-x\bar y|}{|x-y|}
\quad \text{and} \quad \forall x,y \in \Hb,~ G_\Hb(x,y) = \frac{1}{\pi} \log \frac{|x-\bar y|}{|x-y|}.
\end{equation}
Suppose that $\partial D$ is smooth. Then for $x\in D$ and $z\in \partial D$, the Poisson kernel is given by
$
H_D(x,z)=\lim_{\eps\to0} \eps^{-1} G(x, z+\eps \mathbf{n}_z),
$
where $\mathbf{n}_z$ is the inward unit normal vector of $\partial D$ at $z$.
We have
\begin{equation}
\label{E:Poisson}
\forall x \in \D, \forall z \in \partial \D,~
H_\D(x,z) = \frac{1}{2\pi} \frac{1-|x|^2}{|x-z|^2}
\quad \text{and} \quad
\forall x \in \Hb, \forall z \in \partial \Hb,~
H_\Hb(x,z) = \frac{\Im(x)}{\pi|x-z|^2}.
\end{equation}
For $z, w\in \partial D$, let the boundary Poisson kernel be
$
H_D(z,w)=\lim_{\eps\to0} \eps^{-1} H_D(z+\eps \mathbf{n}_z, w).
$
Under this normalisation, we have
\begin{equation}
\label{E:Poisson_boundary}
\forall z,w\in \partial \D,~ H_\D(z,w) = \frac{1}{\pi|z-w|^2}
\quad \text{and} \quad
\forall z,w \in \partial \Hb,~ H_\Hb(z,w) = \frac{1}{\pi(z-w)^2}.
\end{equation}
The Green's function and Poisson kernel are conformally invariant/covariant in the following sense. 
For any conformal map $f : D \to D'$, $x,y \in D$ and $z,w \in \partial D$ such that $\partial D$ (resp. $\partial D'$) is locally smooth near $z$ and $w$ (resp. near $f(z)$ and $f(z)$),
\begin{equation}
\label{E:Green_conformal}
G_{D'}(f(x),f(y)) = G_D(x,y),
\end{equation}
\begin{equation}
\label{E:Poisson_conformal}
H_{D'}(f(x),f(z)) = |f'(z)|^{-1} H_D(x,z)
\quad \text{and} \quad
H_{D'}(f(z),f(w)) = |f'(z)|^{-1}|f'(w)|^{-1} H_D(z,w).
\end{equation}

\medskip

We will also use the explicit expression of the Poisson kernel in a horizontal strip $\R + i (0,\pi h)$ where $h>0$: for all $x \in \R + i (0,\pi h)$ and $x_0 \in \R$,
\begin{align}
\label{E:Poisson_strip1}
H_{\R + i (0,\pi h)}(x,x_0) & = \frac{1}{2\pi h} \frac{\sin(\Im~x/h)}{\cosh((\Re~x - x_0)/h) - \cos (\Im~x/h)}, \\
H_{\R + i (0,\pi h)}(x,x_0+ i \pi h) & = \frac{1}{2\pi h} \frac{\sin(\Im~x/h)}{\cosh((\Re~x - x_0)/h) + \cos (\Im~x/h)}.
\end{align}
See for instance \cite{widder1961functions}.
(Thanks to the change of coordinate formula \eqref{E:Poisson_conformal}, this amounts to finding an explicit expression for a conformal transformation mapping the strip to the upper half plane, or the unit disc.)
The boundary Poisson kernel is given by: for all $x_1, x_2 \in \R$,
\begin{align}
H_{\R + i (0,\pi h)}(x_1,x_2) = H_{\R + i (0,\pi h)}(x_1+i\pi h,x_2+i\pi h) & = \frac{1}{2\pi h^2} \frac{1}{\cosh((x_1-x_2)/h)-1},\\
H_{\R + i (0,\pi h)}(x_1+i\pi h,x_2) = H_{\R + i (0,\pi h)}(x_1,x_2+i\pi h) & = \frac{1}{2\pi h^2} \frac{1}{\cosh((x_1-x_2)/h)+1}.\label{E:Poisson_strip4}
\end{align}
Moreover, for all $x \in \R + i (0,\pi h)$,
\begin{equation}\label{E:Poisson_integral}
    \int_{\R} H_{\R + i (0,\pi h)}(x,x_0) \,\d x_0 = 1 - \frac{\Im(x)}{\pi h}.
\end{equation}
Indeed, the left hand side corresponds to the probability that a 1D Brownian motion starting at $\Im(x)$ hits 0 before $\pi h$. The latter turns out to be equal to the right hand side term of \eqref{E:Poisson_integral}.

\subsection{Hulls and conformal maps}

We call a bounded subset $A\subset \Hb$ an $\Hb$-hull if $A=\Hb \cap \ol A$ and $\Hb \setminus A$ is simply connected. For every $\Hb$-hull $A$, there is a unique conformal transformation $g_A$ from $\Hb\setminus A$ onto $\Hb$ with 
\begin{align*}
\lim_{z\to\infty} [g_A(z)-z]=0.
\end{align*}
We recall the following result:
\begin{lemma}[Proposition 3.38, \cite{MR2129588}]\label{lem:imaginary}
Suppose that $A$ is a $\Hb$-hull, $(B_t)_{t \geq 0}$ is a planar Brownian motion, and let $\tau$ be the smallest $t$ such that $B_t\in \Rb \cup A$. Then for all $z\in \Hb\setminus A$,
\begin{align*}
\Im(z)=\Im(g_A(z)) - \Eb_z[\Im(B_\tau)].
\end{align*}
\end{lemma}

The following lemma is the main result of this section. It gives a quantitative control on the distance between $g_A$ and the identity map, for points not too close to the set $A$. 
For $\eps>0$, let
\begin{equation}\label{E:rectangles}
    R_1 = \{ x + i y: x \in [-1,\eps], y \in [0,\eps]\} \qquad \text{and} \qquad R_2 = \{ x + i y, x \in [-1-\eps,2\eps], y \in  [0,2\eps]\}.
\end{equation}

\begin{lemma}\label{L:hull} 
There exists $c<\infty$ such that if $\eps\le 1/10$ and $A$ is an $\Hb$-hull contained in the rectangle $R_1$, then for all $z\in\Hb\setminus R_2$, we have
\begin{equation}\label{eq:hull}
|g_A(z) - z| \le c\eps |\log \eps|.
\end{equation}
\end{lemma}
\begin{proof}
Many ideas of this proof come from \cite[Proposition 3.50]{MR2129588}.
Let $h(z)=z-g_A(z)$ and $v(z)= \Im(h(z))$. 
Note that $v(z)$ is harmonic. By Lemma~\ref{lem:imaginary}, we have $v(z)\ge 0$. Let $r= \dist(z, \Rb \cup A)$. For all $\zeta \in B(z, r)$, we have 
\begin{align*}
v(\zeta) = \int_{\partial B(z,r)} v(w) F(\zeta,w) |d w|, \qquad
\text{where}\quad F(\zeta, w)=\frac{1}{2\pi r}\Re \left(\frac{w+\zeta-2z}{w-\zeta}\right).
\end{align*}
Letting $\zeta=x+iy$, we have
$\left| \partial_x F(\zeta, w)\right| _{\zeta=z} \le 1/(\pi r^2)$ and $\left| \partial_y F(\zeta, w)\right| _{\zeta=z} \le 1/(\pi r^2)$. Therefore 
\begin{align*}
|\partial_x v(z)| \le \int_{\partial B(z,r)} v(w) (\pi r^2)^{-1} |d w| = 2 v(z)/r,
\end{align*}
and similarly $|\partial_x v(z)| \le 2 v(z)/r$.
This implies $|h'(z)| \le 2\sqrt{2} v(z) /r$. Since $h(z)\to 0$ as $y\to \infty$, we have
\begin{align}\label{eq:boundh}
|h(z)|\le \int_{y}^\infty |h'(x+iu)| du \le 2\sqrt{2} \int_y^\infty v(x+iu) \dist(x+iu, \Rb\cup R)^{-1} du.
\end{align}
We now bound the value of $v(x+iu)$ using Lemma~\ref{lem:imaginary}. We distinguish a few cases. There exists a constant $c>0$ such that the following holds.
\begin{enumerate}
\item If $u \ge 1$, then $v(x+iu) \le c\eps /u$.
\item If $x\in (-\infty, -1-\eps]$, then let $r=-1-x\ge \eps$.  We have $v(x+iu) \le c\eps u/r^2$ for $u\in (0, r^2]$ and $v(x+iu) \le \eps $ for $u\in ( r^2,1]$.
\item If $x\in [2\eps, \infty)$, then let $r=x-\eps\ge \eps$. We have $v(x+iu) \le c\eps u/r^2$ for $u\in (0, r^2]$ and $v(x+iu) \le \eps $ for $u\in ( r^2,1]$.
\item If $x\in [-1-\eps, 2\eps]$ and $u\in [2\eps, 1]$, then  $v(x+iu) \le \eps$.
\end{enumerate}
Now, if $x\in (-\infty, -1-\eps]$, then let $r=-1-x\ge \eps$. We plug the bounds 1 and 2 into \eqref{eq:boundh}, and get that there exist $c_1, c_2>0$, such that for all $z\in \Hb$ with $\Re z=x \in  (-\infty, -1-\eps]$,
\begin{align*}
|h(z)| \le c_1 \int_0^{r^2} \eps /r^2 du + c_1 \int_{r^2}^1 \eps u^{-1} du  + c_1 \int_1^\infty \eps u^{-2} du \le c_2 \eps  |\log \eps|.
\end{align*}
If $x\in [2\eps, \infty)$, then let $r=x-\eps \ge \eps$. We plug the bounds 1 and 3 into \eqref{eq:boundh}, and get similarly for all $z\in \Hb$ with $\Re z=x \in  [2\eps, \infty)$ that $|h(z)| \le c_2 \eps  |\log \eps|.$
If $x\in [-1-\eps, 2\eps]$ and $u \ge 2\eps$, then we plug the bounds 1 and 4 into \eqref{eq:boundh}, and get for all $z$ with $\Re z=x \in [-1-\eps, 2\eps]$ and $\Im z \ge 2\eps$,
\begin{align*}
|h(z)|\le c_1 \int_{2\eps}^1 \eps/(u-\eps) du + c_1 \int_1^\infty \eps u^{-1}(u-\eps)^{-1} du \le c_2 \eps  |\log \eps|.
\end{align*}
This completes the proof.
\end{proof}

We now state a consequence of Lemma \ref{L:hull} that we will use later. Recall that the rectangles $R_1$ and $R_2$ are defined in \eqref{E:rectangles}.

\begin{corollary}\label{C:approx}
There exists $c_1>0$ such that the following holds. Let $\eps\in (0, 1/10)$ and suppose that $\gamma$ is a curve parametrized by $[0,\tau]$ with  $\gamma(0)=0$, $\gamma(\tau)\in [-1, -1+i \eps]$ and $\gamma([0,\tau]) \subset R_1$. 
Let $f$ be the unique conformal map from $\Hb\setminus \gamma([0, \tau])$ onto $\Hb$ with $f(0^+)=0$, $f(\gamma(\tau))=-1$ and $f(\infty)=\infty$. Then for all $z\in \Hb\setminus R_2$, we have
\begin{align*}
|z-f(z)|\le c_1 \eps |\log\eps|.
\end{align*}
\end{corollary}

\begin{proof}
Let $g$ be the unique conformal map from $\Hb \setminus \gamma([0, \tau])$ onto $\Hb$ with $\lim_{z\to\infty}[g(z)-z]=0$. 
By Lemma~\ref{L:hull}, we know that there exists a constant $c>0$, such that for all $\eps\le 1/10$, almost surely on the event $E_\eps$, for all $z\in \Hb \setminus R_2$, we have 
\begin{align}\label{eq:gzz}
|g(z) -z| \le c \eps |\log \eps|.
\end{align}
Note that $f(z)= (g(z)-g(0^+))/(g(0^+)-g(\gamma(\tau)))$. To conclude, it suffices to show that  there exists $c>0$ such that
\begin{align}\label{eq:f_suff}
|g(0^+)| \le c_2 \eps |\log \eps| \qquad \text{and} \qquad |g(\gamma(\tau)) - (-1)| \le c \eps |\log \eps|.
\end{align}
Note that 
\begin{align}\label{eq:g(0)}
|g(0^+) |\le |g(2\eps) -2\eps| + 2\eps + |g(2\eps) - g(0^+)|
\end{align}
Let $B$ be a Brownian motion and let $\sigma$ be the first time $t$ that $B_t$ hits $\Rb \cup \gamma([0,\tau])$. Let $\sigma_0$ be the first time $t$ that $B_t$ hits $\Rb$. Then
\begin{align}
\notag
g(2\eps) - g(0^+)=&\lim_{y\to \infty} \pi y\Pb^{iy} [g(B_\sigma) \in [g(0^+), g(2\eps)]]= \lim_{y\to \infty}  \pi y \Pb^{g^{-1}(iy)}[B_\sigma \in [0, 2\eps]]\\
\label{eq:2eps}
=& \lim_{y\to \infty}  \pi y \Pb^{iy}[B_\sigma \in [0, 2\eps]] \le  \lim_{y\to \infty}  \pi y \Pb^{iy}[B_{\sigma_0} \in [0, 2\eps]] = 2\eps.
\end{align}
Plugging \eqref{eq:gzz} and \eqref{eq:2eps} into \eqref{eq:g(0)}, we get $|g(0^+)| \le c_1 \eps |\log \eps|$ for some $c_1>0$.
We also have
\begin{align}\label{eq:g}
|g(\gamma(\tau)) - (-1)| \le |g(-1-\eps)-(-1-\eps)| +\eps + |g(-1-\eps)- g(\gamma(\tau))|.
\end{align}
Let $I$ be the union of $[-1-\eps, 0]$ together with the left hand-side of $\gamma([0,\tau])$. We have
\begin{align*}
g(\gamma(\tau))-g(-1-\eps) = \lim_{y\to \infty}  \pi y \Pb^{iy}[B_\sigma \in I].
\end{align*}
Let $\sigma_1$ be the first time $t$ that $B_t$ hits $\Rb \cup [-1, \gamma(\tau)] \cup \gamma([0,\tau])$. Let $I_1$ be the union of $[-1-\eps, -1]$ with the left-hand side of $[-1, \gamma(\tau)]$. Then $\Pb^{iy}[B_\sigma \in I] \le \Pb^{iy}[B_{\sigma_1} \in I_1]$. Let $\sigma_2$  be the first time $t$ that $B_t$ hits $\Rb \cup [-1, \gamma(\tau)]$. Then $ \Pb^{iy}[B_{\sigma_1} \in I_1] \le \Pb^{iy}[B_{\sigma_2} \in I_1]$, hence 
\begin{align*}
g(\gamma(\tau))-g(-1-\eps) \le \lim_{y\to \infty}  \pi y \Pb^{iy}[B_{\sigma_2} \in I_1] = \wt g(\gamma(\tau))- \wt g(-1-\eps),
\end{align*}
where $\wt g$ is the unique conformal map from $\Hb\setminus [-1, \gamma(\tau)]$ onto $\Hb$ with $\lim_{z\to\infty}[\wt g(z)-z]=0$. Let $\delta=|\gamma(\tau)| \in [0, \eps]$, then $\wt g(z)= ((z+1)^2 + \delta^2)^{1/2}-1$, and $ \wt g(\gamma(\tau))-\wt g(-1-\eps) =\sqrt{\eps^2+\delta^2} \le 2 \eps$. This implies $g(\gamma(\tau))-g(-1-\eps)\le 2\eps$.
Plugging it back into \eqref{eq:g}, adjusting $c_1$ if necessary, we get
$|g(\gamma(\tau)) - (-1)| \le c_1 \eps |\log \eps|.$
This completes the proof of \eqref{eq:f_suff} and thus of Corollary \ref{C:approx}.
\end{proof}

\section{Conditioning a loop on a portion of its outer boundary}\label{S:conditioning}

Let us start by recalling a result that we already alluded to.

\begin{proposition}[Proposition 3.6, \cite{MR3901648}]\label{prop:decomp}
There exists a probability measure $\pint$ on the space of loops $\{\Pc \in \Lc: \Pc \subset \ol \D\}$ which is invariant under any conformal map of the unit disc and such that the following holds. 
For $\nu_\SLE$-almost all $\gamma$, the following holds. Let $\varphi_\gamma : \D \to \inte(\gamma)$ be a conformal transformation. Then the push forward of $\pint$ by $\varphi_\gamma$ is distributed according to $\loopm (\cdot \vert \gamma)$ \eqref{E:decompo_loopmeasure}.
\end{proposition}

The probability law $\pint$ was denoted by $\P^\inte$ in \cite{MR3901648}.
Informally, 
it corresponds to the law of a Brownian loop $\Pc \sim \loopm$ conditioned on the event that its outer boundary agrees with the unit circle.
Let $\varphi :\D \to \Hb$ be a conformal map. 
We will denote by $\P^\wired_{\Hb,\R}$ the push forward of $\pint$ by $\varphi$. By conformal invariance of $\pint$, $\P^\wired_{\Hb,\R}$ does not depend on our specific choice of conformal map $\varphi$.

In this section, and inspired by \cite{MR3901648}, we will explain how to explore only partially the outer boundary of a Brownian loop. We will define in this way a loop wired only on a portion of the boundary that satisfies one-sided conformal restriction with parameter $\alpha = 5/8$; see Lemma \ref{lem:wired}. We will then show that this partially wired Brownian loop  can be decomposed as the concatenation of a locally finite point process of Brownian excursions. This will be the content of Proposition \ref{prop:pp} which is the main result of this section. We will carry out this procedure through the Brownian bubble measure, introduced in \cite{MR1992830} and \cite{MR2045953}.

\paragraph{Brownian bubble measure}

The Brownian bubble measure in the upper half plane $\Hb$ rooted at $0$ is defined as
\begin{align}\label{eq:bub}
\bubm{0}_{\Hb} = \lim_{\eps\to 0} \frac{\pi}{\eps} H_\Hb(i\eps, 0) \Pb_\Hb^{i \eps,0},
\end{align}
where $ \Pb_\Hb^{i \eps,0}$ is the probability measure on Brownian excursions in $\Hb$ from $i\eps$ to $0$ and $H_\Hb(i\eps, 0) $ is the Poisson kernel in $\Hb$ from $i\eps$ to $0$. For $x,y \in \Rb$, we also denote by $\bubm{x+iy}_{\Hb +iy}$ the bubble measure in $\Hb+iy$ rooted at $x+iy$ (obtained from $\bubm{0}_{\Hb}$ by translation).
The following decomposition (see \cite[Proposition 7]{MR2045953}) relates the Brownian bubble measure to the Brownian loop measure
\begin{align}\label{eq:bub_loop}
\loopm=\frac{1}{\pi}\int_{-\infty}^{\infty} \int_{-\infty}^{\infty} \bubm{x+iy}_{\Hb +iy} dxdy.
\end{align}
It was shown in \cite[Section 7.1]{MR1992830} that if $\Pc$ is distributed according to $\bubm{0}_{\Hb} $, then $\out(\Pc)$ is distributed according to a constant times of SLE$_{8/3}$ bubble measure. More precisely, we define the SLE$_{8/3}$ bubble measure  as
\begin{align*}
\nu^{0,\mathrm{bub}}_\Hb:= \lim_{\eps\to 0} \eps^{-2} \Pb_{0, \eps}^{\mathrm{SLE}},
\end{align*}
where $ \Pb_{0, \eps}^{\mathrm{SLE}}$ is the probability measure on SLE$_{8/3}$ curves in $\Hb$ from $0$ to $\eps$. 
Note that  $\nu^{0,\mathrm{bub}}_\Hb$ and $\bubm{0}_{\Hb} $ are infinite measures, but are finite when we restrict to big loops.

\medskip

Suppose that $\Pc_1$ is a Brownian bubble sampled according to the probability measure obtained by renormalizing $\bubm{0}_{\Hb} $ restricted to loops that intersect $-1 + i\R$:
\begin{equation}\label{E:Pc1}
    \Pc_1 \sim \bubm{0}_{\Hb}(\d \Pc) \indic{\Pc \cap (-1 + i\R) \neq \varnothing} / \bubm{0}_{\Hb} ( \{ \Pc: \Pc \cap (-1 + i\R) \neq \varnothing \} ).
\end{equation}
We regard $\gamma:=\out(\Pc_1)$ as a curve from $0$ to $0$ oriented clockwise and we define
\begin{equation}
    \label{E:tau}
    \tau := \inf \{t>0: \gamma(t) \in -1+i \R \}.
\end{equation}
Let $f$ be the unique conformal map from $\Hb\setminus \gamma([0, \tau])$ onto $\Hb$ with $f(0^+)=0$, $f(\gamma(\tau))=-1$ and $f(\infty)=\infty$.
\begin{lemma}\label{lem:wired}
The law of $f(\Pc_1)$ is invariant under all conformal maps from $\Hb$ onto itself that fixes $-1$ and $0$. Moreover, it satisfies one-sided conformal restriction with parameter $\alpha = 5/8$
which means that, for all $A \subset \ol\Hb$ s.t. $\Hb \setminus A$ is simply connected and s.t. the distance from $A$ to $[0,1]$ is positive,
\[
\P(f(\Pc_1) \cap A = \varnothing) = \phi_A'(0)^\alpha \phi_A'(-1)^\alpha,
\]
where $\phi_A : \Hb \setminus A \to \Hb$ is any conformal map fixing $-1$ and $0$.
\end{lemma}

We will denote by $\P_{\Hb,[-1,0]}^\wired$ the law of $f(\Pc_1)$.
We define $\Pb^{\mathrm{wired}}_{\Hb, I}$ for any interval $I \neq \R$ as the image of $\Pb^{\mathrm{wired}}_{\Hb, [-1,0]}$ under a conformal map from $\Hb$ onto itself that maps $[-1,0]$ to $\ol I$. 

\begin{proof}
By definition of the SLE$_{8/3}$ bubble measure, the image under $f$ of the part of $\gamma$ after $\tau$ is distributed as a chordal SLE$_{8/3}$ curve $\wt \gamma$ in $\Hb$ from $-1$ to $0$. The law of $\wt \gamma$ is  invariant under all conformal maps from $\Hb$ onto itself that fixes $-1,0$, and satisfies one-sided conformal restriction with parameter $5/8$ (see \cite[Section 6]{MR1992830}). By Proposition~\ref{prop:decomp} and \eqref{eq:bub_loop}, we know that $f(\Pc_1)$ has the same law as the following random object: We first sample a chordal SLE$_{8/3}$ curve $\wt \gamma$ in $\Hb$ from $-1$ to $0$. In the bounded connected component $C(\wt\gamma)$ of $\Hb\setminus \wt\gamma$, we choose a random point $z$, say uniformly in $C(\wt\gamma)$. Let $\varphi$ be the conformal map from $\Db$ onto $C(\wt\gamma)$ with $\varphi(0)=z$ and $\varphi(1)=0$. Let $K$ be an independent set with law $\pint$. Then $\varphi(K)$ has the same law as $f(\Pc_1)$. Due to the conformal invariance of $\pint$ and $\wt \gamma$, we can deduce that the law of $\varphi(K)$ is invariant under all conformal maps from $\Hb$ onto itself that fixes $-1,0$. 
\end{proof}

We can now state and prove the main result of this section.

\begin{proposition}\label{prop:pp}
Let $\Pc$ be distributed as  $\Pb^{\mathrm{wired}}_{\Hb, [0,1]}$. Then $\Pc$ is the concatenation of a locally finite point process of Brownian excursions in $\Hb$ with endpoints in $[0,1]$.
\end{proposition}

\begin{remark}
By ``locally finite'', we mean that for any $\delta>0$, there are a.s.\ finitely many excursions with diameter $\ge \delta$. By ``point process'', we mean that given the pairs of endpoints of the excursions, the excursions are distributed as independent Brownian excursions in $\Hb$ connecting those endpoints.
By ``concatenation'', we mean that these excursions come with a natural cyclic ordering allowing us to glue them together to recover the unrooted loop modulo time parametrisation. 
Note however that this cycling ordering is \textit{a priori} not measurable w.r.t. the point process of excursions.
This proposition is analogous to \cite[Lemma 9]{qian2018} where a similar result is proved for a Brownian loop soup wired on the entire boundary.
\end{remark}
\begin{proof}[Proof of Proposition~\ref{prop:pp}]
\begin{figure}
\centering
\includegraphics[trim={0 0 0 1.7cm},clip, width=\textwidth]{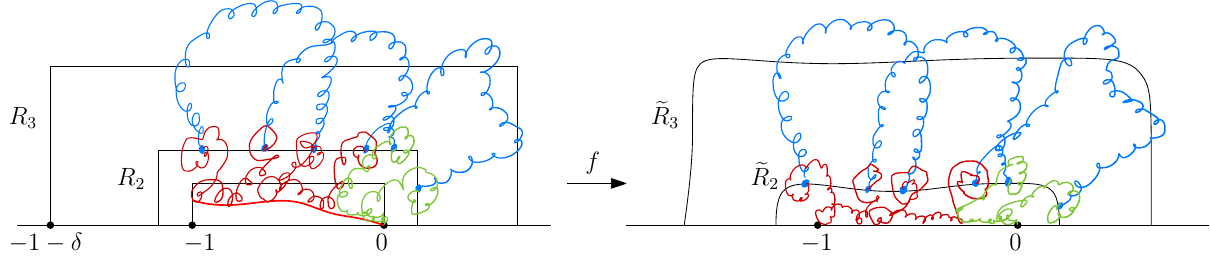}
\caption{On the left, we depict the decomposition of a Brownian bubble $\Pc$, where the event $E_\eps$ occurs. The excursions in (1) (2) (3) are respectively drawn in blue, red and green. The picture on the left is mapped by $f$ to the picture on the right.}
\label{fig:decomposition}
\end{figure}
The ``locally finite'' property immediately follows from the fact that that $\Pc$ is a continuous loop.
To prove the ``point process'' property, we first make the following decomposition for a loop $\Pc_1$ distributed according to \eqref{E:Pc1}. See Figure~\ref{fig:decomposition} for an illustration. Fix $\delta>2\eps>0$. As in \eqref{E:rectangles}, let $R_2$ be the rectangle with corners $-1-\eps$ and $2\eps+ 2\eps i$ and let $R_3$ be the rectangle with corners $-1-\delta$ and $\delta+ i\delta$.
Then $\Pc_1$ can be decomposed as the concatenation of 
\begin{enumerate}[(1)]
\item a finite number of excursions in $\Hb\setminus R_2$ which have both endpoints on $\partial R_2$, and intersect $\partial R_3$;
\item excursions in $R_3$ that connect the endpoints of the excursions in (1);
\item an excursion in $R_3$ from $0$ to an endpoint of an excursion in (1), and an excursion in $R_3$ from an endpoint of an excursion in (1) back to $0$.
\end{enumerate}
The strong Markov property of Brownian path measures ensures that given the endpoints, the excursions in (1) are distributed as independent Brownian excursions in $\Hb\setminus R_2$ that are conditioned to intersect $\partial R_3$.

Lemma~\ref{lem:wired} and its proof allow us to decompose $\Pc_1$ into three independent random objects: the curve $\gamma([0,\tau])$ (see \eqref{E:tau}), the curve $\wt \gamma$ distributed as an SLE$_{8/3}$ from $-1$ to $0$, and $\Kc$ distributed as $\pint$. As before, we will denote by $f$ the unique conformal map from $\Hb\setminus \gamma([0, \tau])$ onto $\Hb$ with $f(0^+)=0$, $f(\gamma(\tau))=-1$ and $f(\infty)=\infty$.
Consider the event
\begin{equation}
    \label{E:event_Eeps}
    E_\eps = \{ \gamma([0,\tau]) \subset R_1 \}.
\end{equation}
This event is determined by (2) and (3), and is independent from the excursions in (1) (conditionally on their endpoints which are determined by (2) and (3)).

The event $E_\eps$ only depends on $\gamma([0,\tau])$, not on $\wt \gamma$ nor on $\Kc$.
Because $f(\Pc_1)$ is independent from $\gamma([0,\tau])$, the law of $f(\Pc_1)$ conditioned on $E_\eps$ is the same as its unconditioned law, which is  $\Pb^{\mathrm{wired}}_{\Hb, [-1,0]}$.
Now, let $\wt R_2:=f(R_2)$ and $\wt R_3=f(R_3)$ which are independent of $f(\Pc_1)$ (they only depend on $\gamma([0,\tau])$).
We have therefore proved the following statement:  Suppose $\Pc$ is sampled w.r.t.\ the law $\Pb^{\mathrm{wired}}_{\Hb, [-1,0]}$. Suppose $\wt R_2$ and $\wt R_3$ are sampled according to $f(R_2)$ and $f(R_3)$ conditioned on $E_\eps$, independently from $\Pc$.
Then $\Pc$ can be decomposed into three parts:
\begin{enumerate}[(1')]
\item a finite number of excursions in $\Hb\setminus \wt R_2$ which have both endpoints on $\partial \wt R_2$, and intersect $\partial \wt R_3$;
\item excursions in $\wt R_3$ that connect the endpoints of the excursions in the previous bullet point;
\item an excursion in $\wt R_3$ from $0$ to an endpoint of an excursion in (1), and an excursion in $\wt R_3$ from an endpoint of an excursion in (1) back to $0$.
\end{enumerate}
Moreover, let $\Ec_1$ be the set of excursions in (1'). Conditionally given the endpoints of the excursions, $\Ec_1$ is distributed as a set of independent Brownian excursions in $\Hb\setminus \wt R_2$ connecting their endpoints on $\partial \wt R_2$ and conditioned to intersect $\partial \wt R_3$.

Let $\Ec$ be the collection of excursions away from $[-1,0]$ induced from $\Pc$. 
Let $\Ec(R_3)$ be the collection of excursions in $\Ec$ that intersect $\partial R_3$.
 Note that both $\Ec_1$ and $\Ec(R_3)$ are finite sets, due to the ``local finiteness''. 
  Corollary~\ref{C:approx} ensures that $\partial \wt R_2$ converges a.s.\ to the segment $[-1,0]$ and $\partial \wt R_3$ converges a.s.\ to $\partial R_3$, both w.r.t.\ the Hausdorff distance, as $\eps\to0$. Therefore,
$\Ec_1$ converges to $\Ec(R_3)$  as $\eps\to 0$, in the following sense: For $\eps$ small enough, there is a bijection between $\Ec_1$ and $\Ec(R_3)$. Moreover, each excursion in  $\Ec_1$ is a subset of the corresponding excursion in $\Ec(R_3)$ and it converges to a corresponding excursion in $\Ec(R_3)$ w.r.t.\ the Hausdorff distance. This implies that given the endpoints of the excursions, $\Ec(R_3)$ is distributed as a  set of independent Brownian excursions in $\Hb$ connecting their endpoints on $[-1,0]$ and conditioned to intersect $\partial R_3$. Since this is true for any $\delta>0$, it implies the proposition.
 \end{proof}

\begin{notation}\label{N:exc}
A loop $\Pc$ sampled according to $\pint$ is the concatenation of countably many Brownian excursions in $\D$ with endpoints in $\partial \D$. Let us denote by $\Ec_{\partial \D}$ such a collection of excursions. For each excursion $e \in \Ec_{\partial \D}$, let $a(e)$ and $b(e)$ be its end points. For symmetry purposes, $a(e)$ is randomly chosen to be one of the two endpoints with equal probability 1/2 and $b(e)$ is the other endpoint.
Similarly, if $\Pc \sim \P^\wired_{\Hb,\R}$ (resp. $\Pc \sim \P^\wired_{\Hb,I}$ for some interval $I \subset \R$), then $\Pc$ is the concatenation of countably many excursions in $\Hb$ that we denote by $\Ec_\R$ (resp. $\Ec_I$). 
\end{notation}

\medskip

\paragraph{Consequences of Lemma \ref{lem:wired} and Proposition \ref{prop:pp}.}
We now list a few important consequences of the Lemma \ref{lem:wired} and Proposition \ref{prop:pp} that were derived in \cite{JLQ23b}.

\begin{corollary}\label{C:JLQ}
Let $\alpha = 5/8$ be the conformal restriction parameter.
    For all $r>0$,
    \begin{equation}
        \label{E:C_JLQ1}
    \E \Big[ \sum_{e \in \Ec_{\R^-}} \frac{(a(e)-b(e))^2 r^2}{(r-a(e))^2(r-b(e))^2} \Big]
    = \alpha,
    \end{equation}
    where $\Ec_{\R^-}$ is the collection of excursions in $\Hb$ with endpoints in $\R^-$ defined in Notation \ref{N:exc}.
    For all $0<r_1<r_2$,
    \begin{align}
    \notag
        & 2 \E \Big[ \sum_{e \in \Ec_{\R^-}} \frac{(a(e)-b(e))^2}{(r_1-a(e))^2(r_2-b(e))^2(r_1-r_2)^2} \Big]
        + \E \Big[ \sum_{e_1 \neq e_2 \in \Ec_{\R^-}} \prod_{j=1,2} \frac{(a(e_j)-b(e_j))^2}{(r_j-a(e_j))^2(r_j-b(e_j))^2} \Big] \\
        & = \frac{\alpha}{r_1^2(r_1-r_2)^2} + \frac{\alpha}{r_2^2(r_1-r_2)^2} + \frac{\alpha(\alpha-1)}{r_1^2 r_2^2}.\label{E:C_JLQ2}
    \end{align}
\end{corollary}

\begin{proof}
    By Lemma \ref{lem:wired} and Proposition \ref{prop:pp}, $\Ec_{\R^-}$ is a locally finite point process of Brownian excursions and its filling satisfies one-sided conformal restriction with exponent $\alpha$. As shown in \cite{JLQ23b} (see in particular \cite[Remark 6.6]{JLQ23b}), these two properties yield the identities \eqref{E:C_JLQ1} and \eqref{E:C_JLQ2}.
\end{proof}

\begin{corollary}\label{C:interval_excursion}
    There exists $C>0$ such that the following holds. Let $I_1, I_2 \subset \R$ be any intervals with respective lengths $r_1$ and $r_2$. Then,
    \begin{equation}
        \label{E:C_interval_excursion1}
    \E \Big[ \sum_{\substack{e \in \Ec_\R\\a(e) \in I_1, b(e) \in I_2}} (a(e) - b(e))^2 \Big] \leq C r_1 r_2.
    \end{equation}
    Moreover, if in addition $I_1 \subset I_2$, then 
    \begin{equation}
        \label{E:C_interval_excursion2}
    \E \Big[ \Big( \sum_{\substack{e \in \Ec_\R\\a(e) \in I_1, b(e) \in I_2}} (a(e) - b(e))^2 \Big)^2 \Big] \leq C r_1 r_2^3.
    \end{equation}
\end{corollary}

\begin{proof}
\eqref{E:C_interval_excursion1} is the analogue of \cite[Lemma 6.20]{JLQ23b} and follows from \eqref{E:C_JLQ1}. See \cite[Section 6.4]{JLQ23b} for details.
Moving to the proof of \eqref{E:C_interval_excursion2}, we first notice that, by scaling and translation, the left hand side of \eqref{E:C_interval_excursion2} is equal to
\[
r_2^4 \E \Big[ \Big( \sum_{\substack{e \in \Ec_\R\\a(e) \in I_1', b(e) \in [0,1]}} (a(e) - b(e))^2 \Big)^2 \Big]
\]
where $I_1' \subset [0,1]$ is an interval of length $r_1/r_2$. By the analogue of \cite[Lemma 6.19]{JLQ23b} which is derived from \eqref{E:C_JLQ2}, the above expectation is at most $C r_1/r_2$ which concludes the proof.
\end{proof}

\begin{corollary}\label{C:upper_bound_2point}
There exists $C>0$, such that for all $x, y \in \Hb$,
\[
\limsup_{r \to 0} |\log r|^2 \Prob{ \exists e, e' \in \Ec_\R, e \cap B(x,r) \neq \varnothing, e' \cap B(y,r) \neq \varnothing } \leq C \max(1, G_\Hb(x,y)).
\]    
\end{corollary}

\begin{proof}
This result is the analogue of \cite[Proposition 6.4]{JLQ23b} and follows from Corollary \ref{C:interval_excursion}. See \cite[Section 6.5]{JLQ23b} for details.
\end{proof}

\section{Proof of Theorems \ref{T:intro} and \ref{T:main}}\label{S:proof_main}

The goal of this section is to prove our main result, Theorem \ref{T:intro} and its analogue Theorem \ref{T:main} for a Brownian loop.
We start in Section \ref{SS:reduction} by showing that Theorem \ref{T:main} implies Theorem \ref{T:intro}. We will further reduce the problem by showing that Theorem \ref{T:main} follows from a third version of this result, Theorem \ref{T:level_line} below.

\subsection{Reduction}\label{SS:reduction}

We first state a version of Theorem \ref{T:main} for $\Pc \sim \pint$.
We will need to consider the following assumption on a sequence $(f_\eps)_\eps$ of test functions $f_\eps : \D \to \R$:

\begin{assumption}\label{assumption_feps}
For all $\eps >0$, $f_\eps$ is compactly supported in $\{ x \in \overline{\D}: \d(x,\partial \D) < \eps \}$ and $\int_\D f_\eps =1$. Moreover,
\begin{gather}
\label{E:assumption_feps}
\sup_\eps \int_{\D \times \D} \max(1, -\log|x-y|) f_\eps(x) f_\eps(y) \d x \d y < \infty,\\
\label{E:assumption_feps2}
\text{and} \quad \lim_{\delta \to 0} \limsup_{\eps \to 0} \int_{\D \times \D} \indic{|x-y| < \delta} \max(1, -\log|x-y|) f_\eps(x) f_\eps(y) \d x \d y = 0.
\end{gather}
\end{assumption}

\begin{example}\label{Example}
If $f : \C \to \R$ is a smooth test function compactly supported in $\D$ with $\int f =1$, then one can check that the sequence $f_\eps : x \in \D \mapsto \frac{1}{\eps} f( \frac{1-|x|}{\eps} \frac{x}{|x|} )$, $\eps >0$, satisfies Assumption \ref{assumption_feps}.
For instance, \eqref{E:assumption_feps} boils down to the fact that the integral
\[
\int_{\D \times \D} \max\Big(1, -\log \Big| \frac{x}{|x|} - \frac{y}{|y|} \Big| \Big) f(x) f(y) \d x \d y
\]
is finite
(a logarithmic singularity is integrable in dimension 1).

\end{example}

\begin{theorem}\label{T:level_line}
Consider a loop $\Pc$ sampled according to $\pint$ (see Proposition \ref{prop:decomp}).
\begin{itemize}
\item
(Constant expectation)
For all deterministic test function $f$,
\begin{equation}
\label{E:T_level_line_expectation}
\Eint \Big[ \int f(x) L_x(\Pc) \d x \Big] = \frac{5}{\pi} \int_\D f.
\end{equation}
\item
(Constant boundary conditions) Let $(f_\eps)_\eps$ be a sequence of deterministic test functions satisfying Assumption~\ref{assumption_feps}. Then
\begin{equation}
\label{E:T_level_line_bc}
\int f_\eps(x) L_x(\Pc) \d x \xrightarrow[\eps \to 0]{L^1(\pint)} \frac{5}{\pi}.
\end{equation}
\end{itemize}
\end{theorem}

The proof of Theorem \ref{T:main} assuming Theorem \ref{T:level_line} will essentially follow from the following lemma. Together with Example \ref{Example}, this lemma also gives a wide family of examples of test functions satisfying Assumption \ref{assumption_feps_main}.

\begin{lemma}\label{L:equivalence_assumptions}
    For $\nu_\SLE$-almost all $\gamma$ and any conformal map $\varphi_\gamma : \inte(\gamma) \to \D$, there exist a constant $c=c(\gamma) \in (0,\infty)$ and a deterministic exponent $h>0$ such that the following holds.
    \begin{itemize}
        \item
    If $(f_\eps)_\eps$ is a sequence of test functions satisfying Assumption \ref{assumption_feps_main}, then $(\tilde f_\eps)_\eps$ satisfies Assumption~\ref{assumption_feps} where, for $\eps >0$,
    \begin{equation}
        \label{E:L_equivalence1}
    \tilde f_\eps \quad : \quad y \in \overline{\D} \quad \longmapsto \quad f_{(c \eps)^{1/h}}(\varphi_\gamma^{-1}(y)) |(\varphi^{-1})'(y)|^2. 
    \end{equation}
        \item
    Conversely, if $(\tilde f_\eps)_\eps$ is a sequence of test functions satisfying Assumption \ref{assumption_feps}, then $(f_\eps)_\eps$ satisfies Assumption \ref{assumption_feps_main} where, for $\eps>0$,
    \begin{equation}
    f_\eps \quad : \quad x \in \overline{\inte(\gamma)} \quad \longmapsto \quad \tilde f_{(c \eps)^{1/h}}(\varphi_\gamma(x)) |\varphi_\gamma'(x)|^2.
    \end{equation}
    \end{itemize}
\end{lemma}

\begin{proof}[Proof of Lemma \ref{L:equivalence_assumptions}]
The key property is that for $\nu_\SLE$-almost all $\gamma$, $\varphi_\gamma$ and $\varphi_\gamma^{-1}$ are uniformly H\"older continuous: there exists $c = c(\gamma) \in (0,\infty)$ and $h>0$ such that for all $y, y' \in \overline{\D}$ and $x,x' \in \overline{\inte(\gamma)}$,
\begin{equation}\label{E:holder}
|\varphi_\gamma^{-1}(y) - \varphi_\gamma^{-1}(y')| \leq c^{-1} |y-y'|^h
\quad \text{and} \quad
|\varphi_\gamma(x) - \varphi_\gamma(x')| \leq c^{-1} |x-x'|^h.
\end{equation}
Indeed, $\gamma$ is locally an SLE$_{8/3}$ curve and the map uniformising the complement of an SLE$_\kappa$ curve is known to be H\"older continuous when $\kappa < 4$; see \cite{MR2153402} (see also \cite{MR3786302} for the optimal H\"older exponent~$h$).

Now, consider a sequence of test functions $(f_\eps)_\eps$ satisfying Assumption \ref{assumption_feps_main} and define $\tilde f_\eps$ as in \eqref{E:L_equivalence1} with $c$ and $h$ as above.
The inequality \eqref{E:holder} for $\varphi_\gamma$ together with the fact that the support of $f_\eps$ is contained in $\{ x \in \overline{\inte(\gamma)} : \d(x,\gamma) < \eps \}$ implies that $\tilde f_\eps$ is supported in $\{ y \in \overline{\D}: \d(y,\partial \D) < \eps \}$. The fact that $\int_\D \tilde f_\eps = 1$ follows directly from the fact that
$\int_{\inte(\gamma)} f_\eps =1$ together with a change of variable. 
Finally, to check \eqref{E:assumption_feps}, we compute with the help of the change of variables $y=\varphi_\gamma(x)$, $y' = \varphi_\gamma(x')$ and $\tilde \eps = (c \eps)^{1/h}$:
\begin{align*}
    & \sup_\eps \int_{\D \times \D} \max(1, - \log|y-y'|) \tilde f_\eps(y) f_\eps(y') \d y \d y' \\
    & = \sup_{\tilde \eps} \int_{\inte(\gamma) \times \inte(\gamma)} \max(1, -\log |\varphi_\gamma(x) - \varphi_\gamma(x')|) f_{\tilde\eps}(x) f_{\tilde\eps}(x') \d x \d x' \\
    & \leq C \sup_{\tilde\eps} \int_{\inte(\gamma) \times \inte(\gamma)} \max(1, -\log |x-x'|) f_{\tilde\eps}(x) f_{\tilde\eps}(x') \d x \d x'.
\end{align*}
In the last inequality we used that $\varphi_\gamma$ is H\"older.
The above right hand side is bounded by assumption on $(f_\eps)_\eps$. This shows that $(\tilde f_\eps)_\eps$ satisfies \eqref{E:assumption_feps}. \eqref{E:assumption_feps2} is similar. This concludes the proof that $(\tilde f_\eps)_\eps$ satisfies Assumption \ref{assumption_feps}. The reverse direction is analogous.
\end{proof}

We now prove Theorem \ref{T:main} assuming Theorem \ref{T:level_line}.

\begin{proof}[Proof of Theorem \ref{T:main}, assuming Theorem \ref{T:level_line}]
Let $\gamma$ be a continuous simple loop and $\Pc$ be sampled according to $\loopm(\cdot \vert \gamma)$, as in the statement of Theorem \ref{T:main}. Let $\varphi_\gamma : \D \to \inte(\gamma)$ be any fixed conformal map.
Recall that $\Pc^\inte := \varphi_\gamma^{-1}(\Pc)$ is distributed according to $\pint$ and is independent of $\gamma$; see Proposition \ref{prop:decomp}.

Let $f$ be a test function. By doing the change of variables $x = \varphi_\gamma(y)$, we have
\begin{align}\label{E:conf_cov}
    \E \Big[ \int f(x) L_x(\Pc) \d x \Big\vert \gamma \Big]
    =
    \E \Big[ \int_\D f(\varphi_\gamma(y)) \varphi_\gamma^* (L_x(\Pc) \d x)(\d y) \Big\vert \gamma \Big],
\end{align}
where $\varphi_\gamma^*(L_x(\Pc) \d x)$ is the pullback of $L_x(\Pc) \d x$ by $\varphi_\gamma$.
By conformal covariance of Brownian motion, one can show that, almost surely,
\[
\varphi_\gamma^*(L_x(\Pc) \d x) = |\varphi_\gamma'(y)|^2 L_y(\Pc^\inte) \d y.
\]
This is a classical result but also corresponds to the case $\gamma_{\mathrm{GMC}}=0$ of the conformal covariance property of Brownian multiplicative chaos; see \cite{bass1994, jegoBMC, AidekonHuShi2018}.
Because $\Pc^\inte$ is independent of~$\gamma$, we have by Theorem \ref{T:level_line}, \eqref{E:T_level_line_expectation}, that
\[
\E \Big[ \int f(x) L_x(\Pc) \d x \Big\vert \gamma \Big] =
\frac5\pi
\int_\D f(\varphi_\gamma(y)) |\varphi_\gamma'(y)|^2 \d y
= \frac5\pi \int_{\inte(\gamma)} f(x) \d x.
\]
This proves \eqref{E:T_main_expectation}.

Moving to the proof of \eqref{E:T_main_bc}, let $(f_\eps)_\eps$ be a sequence of test functions satisfying Assumption \ref{assumption_feps_main}. 
By Lemma \ref{L:equivalence_assumptions}, the sequence $(\tilde f_\eps)_\eps$ defined in \eqref{E:L_equivalence1} satisfies Assumption \ref{assumption_feps}. We can thus perform the same change of variables as above and use Theorem \ref{T:level_line}, \eqref{E:T_level_line_bc}, to deduce that
$\int f_\eps(x) L_x(\Pc) \d x \to \frac{5}{\pi}$ in $L^1$. This concludes the proof.
\end{proof}

Finally, we prove Theorem \ref{T:intro}, assuming Theorem \ref{T:main}:

\begin{proof}[Proof of Theorem \ref{T:intro}, assuming Theorem \ref{T:main}]
For $t>0$ and $z \in \C$, we denote by $\nu_{z,t}$ the push forward of $\P^{t,z,z}$ by the map $\Pc \in \Lc \mapsto \out(\Pc) \in \Gamma$, where we recall that $\P^{t,z,z}$ is the probability law of a Brownian bridge from $z$ to $z$ with duration $t$ and the spaces $\Lc$ and $\Gamma$ are introduced in Definitions \ref{Def:space_loops} and \ref{Def:space_self-avoiding} respectively. 
Since $\nu_\SLE$ is the push forward of $\loopm$ by the map $\Pc \mapsto \out(\Pc)$ \cite{MR2350053}, and by the identity \eqref{E:loopm}, the probability measures $\nu_{z,t}$ are related to $\nu_\SLE$ by
\[
\nu_\SLE = \int_\C \d z \int_0^\infty \frac{\d t}{t} \frac{1}{2\pi t} \nu_{z,t}.
\]
Let $z\in \C$, $t>0$ and $\gamma$ be a continuous self-avoiding loop. Let $(f_\eps)_\eps$ be a sequence of test functions $f_\eps : \inte(\gamma) \to \R$ satisfying Assumption \ref{assumption_feps_main}. Let $E_{z,t}(\gamma)$ be the event that
\begin{align}\label{eq:f_eps}
\int f_\eps(x) L_x(\Pc) \d x \xrightarrow[\eps \to 0]{} \frac{5}{\pi} \qquad \text{in} \quad \mathrm{L}^1(\P^{t,z,z}(\cdot \vert \gamma)).
\end{align}
By translation and scaling, $\nu_{z,t}(E_{z,t}(\gamma))$ does not depend on $z$ or $t$ and we know by Theorem \ref{T:main} that
\[
\int_\C \d z \int_0^\infty \frac{\d t}{t} \frac{1}{2\pi t} \nu_{z,t}(E_{z,t}(\gamma)^c) = 0.
\]
Hence $\nu_{z,t}(E_{z,t}(\gamma))=1$, that is to say, Theorem \ref{T:intro} holds for $\nu_{z,t}$-almost all $\gamma$, for $\Pc \sim \P^{t,z,z}(\cdot \vert \gamma)$.

We now specify this result to $z=0$ and $t=2$ and cut the Brownian bridge $\Pc = (\Pc_t)_{t \in [0,2]}$ into two pieces $\Pc^1 = (\Pc_t)_{t \in [0,1]}$ and $\Pc^2 = (\Pc_{1+t})_{t \in [0,1]}$. Notice that, for almost all realisations of $\Pc^1$ such that $\out(\Pc^1)$ is a simple loop, conditionally on $\Pc^1$, the probability that $\Pc^2$ stays in $\inte(\out(\Pc^1))$ and at a positive distance to $\out(\Pc^1)$ is positive. On this event, the outer boundaries of $\Pc$ and $\Pc^1$ coincide and their respective local times in the vicinity of $\out(\Pc)$ also agree. Therefore, Theorem \ref{T:intro} also holds for $\Pc^1$. Since $\Pc^1$ is mutually absolutely continuous with respect to a Brownian motion starting at $0$ of duration 1, it concludes the proof.

In general, when we consider a simple curve $\xi \subset \out(\Pc^1)$, conditionally on $\Pc^1$ and $\xi$, there is still a positive probability that $\Pc^2$ remains at a positive distance to $\xi$ and that $\out(\Pc) \supset \xi$. The proof then follows by the same reasoning.
\end{proof}

The rest of this section is dedicated to the proof of Theorem \ref{T:level_line}.
Section \ref{SS:correl} introduces ``correlation functions'' associated to the occupation measure.
The proofs of \eqref{E:T_level_line_expectation} and \eqref{E:T_level_line_bc} will be obtained in Sections \ref{SS:second_first} and \ref{SS:second_moment} respectively.

\subsection{Correlation functions}\label{SS:correl}

In this section, we show that the measures $\Eint[L_x(\Pc) \d x]$ and $\Eint[L_x(\Pc) L_y(\Pc) \d x \d y]$ have densities w.r.t. Lebesgue measures. Inspired by the physics literature, we call these densities correlation functions and denote them using $\langle \cdot \rangle$.
We first consider the one-point case.

\begin{lemma}\label{L:1point}
    The measure $\Eint[L_x(\Pc) \d x]$ possesses a density with respect to Lebesgue measure in $\D$ given by
    \begin{equation}
        \label{E:L_1point1}
    x \in \D \quad \longmapsto \quad
    \frac1\pi \E \Big[ \sum_{e \in \Ec_{\partial \D}} \frac{H_\D(x,a(e)) H_\D(x,b(e))}{H_\D(a(e),b(e))} \Big].
    \end{equation}
    We will denote this density by $x \in \D \mapsto \corint{L_x(\Pc)}$.
    Moreover, for all $x \in \D$, we have
    \begin{equation}
        \label{E:L_1point2}
    \corint{L_x(\Pc)} = \frac1\pi \lim_{r \to 0} |\log r| \pint(\Pc \cap B(x,r) \neq \varnothing).
    \end{equation}
    More generally, for any measurable function $F:\C \times \Lc \to [0,1]$, the measure $\Eint[F(x,\Pc) L_x(\Pc) \d x]$ possesses a density with respect to Lebesgue measure in $\D$ that we will denote by $\corint{F(x,\Pc)L_x(\Pc)}$. It is given by
\begin{equation}
    \label{E:L_1point3}
    x \in \D \quad \longmapsto \quad
    \frac1\pi \E \Big[ \sum_{e \in \Ec_{\partial \D}} \frac{H_\D(x,a(e)) H_\D(x,b(e))}{H_\D(a(e),b(e))} F(x,\Pc_x^e ) \Big],
\end{equation}
where the loop $\Pc_x^e$ is obtained from $\Pc$ by replacing the excursion $e$ by an excursion $e_x$. This excursion $e_x$ is independent of the other excursions conditionally on their endpoints, and has the law of an excursion from $a(e)$ to $b(e)$ conditioned to visit $x$: it is the concatenation of an excursion from $a(e)$ to $x$ and an excursion from $x$ to $b(e)$.
\end{lemma}

\begin{proof}
    \eqref{E:L_1point2} is a consequence of the fact that the occupation measure of $\Pc$ coincides with its Minkowski content in the gauge $r \mapsto \frac1\pi r^2 |\log r|$; see \eqref{E:Minkowski123}.
    For \eqref{E:L_1point1}, a small calculation first gives that for any distinct boundary points $a, b \in \partial \D$, the expectation of the occupation measure of a Brownian excursion from $a$ to $b$ in $\D$ possesses a density w.r.t. Lebesgue measure on $\D$ given by
    \[
    x \in \D \quad \longmapsto \quad
    \frac1\pi \frac{H_\D(x,a) H_\D(x,b)}{H_\D(a,b)}.
    \]
    \eqref{E:L_1point1} is then obtained by Fubini. \eqref{E:L_1point3} is a direct generalisation.
\end{proof}

\textit{A priori} the density $\corint{L_x(\Pc)}$ is only well defined for Lebesgue-almost every $x \in \D$, but the right hand side of \eqref{E:L_1point1} defines it for all $x \in \D$.

The following lemma considers the law $\loopm(\cdot\vert\gamma)$ and relates the associated one-point function to $\corint{\cdot}$.

\begin{lemma}\label{L:1point_gamma}
    For $\nu_\SLE$-almost every $\gamma$, the following holds. Let $F: \inte(\gamma) \times \Lc \to [0,1]$ be a measurable map. The following measure is absolutely continuous w.r.t. Lebesgue measure in $\inte(\gamma)$:
    \[
    \int \loopmeasure(\d \Pc \vert \gamma) F(x,\Pc) L_x(\Pc) \d x = \langle F(x,\Pc) L_x(\Pc) \vert \gamma \rangle^{\rm loop} \d x,
    \]
    where the integral is w.r.t. $\Pc$ and the right hand side is defined by this identity.
    Moreover, if $\varphi_\gamma : \D \to \inte(\gamma)$ is a conformal map, then for all $x \in \inte(\gamma)$,
    \begin{equation}
        \label{E:L_1point4}
    \langle F(x,\Pc) L_x(\Pc) \vert \gamma \rangle^{\rm loop}
    = \corint{F(x, \varphi_\gamma(\Pc)) L_{\varphi_\gamma^{-1}(x)}(\Pc)}.
    \end{equation}
\end{lemma}
We emphasise that on the right hand side of \eqref{E:L_1point4}, $\gamma$ is considered fixed.
\begin{proof}
    Let $\Pc^\inte$ be distributed from $\pint$ and let $\Pc := \varphi_\gamma(\Pc^\inte)$. By Proposition \ref{prop:decomp}, $\Pc$ follows the law $\loopmeasure(\cdot \vert \gamma)$. Let $f: \inte(\gamma) \to [0,\infty)$ be a test function.
    By doing the change of variable $x=\varphi_\gamma(y)$ and conformal covariance of the occupation measure (see \eqref{E:conf_cov}),
    \begin{align*}
        & \int \loopmeasure(\d \Pc \vert \gamma) \int_{\inte(\gamma)} f(x) F(x,\Pc) L_x(\Pc) \d x \\
        & = \Eint \Big[ \int_\D f(\varphi_\gamma(y)) F(\varphi_\gamma(y),\varphi_\gamma(\Pc^\inte))
        |\varphi_\gamma'(y)|^2 L_y(\Pc^\inte) \d y \vert \gamma \Big].
    \end{align*}
    By Lemma \ref{L:1point} and then by undoing the change of variable, this is further equal to
    \begin{align*}
        & \int_\D f(\varphi_\gamma(y)) \corint{F(\varphi_\gamma(y),\varphi_\gamma(\Pc^\inte)) L_y(\Pc^\inte) }
        |\varphi_\gamma'(y)|^2  \d y\\
        & = \int_{\inte(\gamma)} f(x) \corint{F(x,\varphi_\gamma(\Pc^\inte)) L_{\varphi_\gamma^{-1}(x)}(\Pc^\inte) } \d x. 
    \end{align*}
    Since this is true for any test function $f$, this proves the existence of the desired density as well as its identification.
\end{proof}

We now considers two-point functions.

\begin{lemma}\label{L:2point}
The measure $\Eint[L_x(\Pc) L_y(\Pc) \d x \d y]$ possesses a density with respect to Lebesgue measure on $\D \times \D$ given by
\begin{align}\label{E:L_2point}
    (x,y) \in \D \times \D \quad \longmapsto \quad & \frac1\pi G_\D(x,y) \E \Big[ \sum_{e \in \Ec_{\partial \D}} \frac{H_\D(x,a(e)) H_\D(y,b(e)) + H_\D(x,b(e)) H_\D(y,a(e))}{H_\D(a(e),b(e))} \Big] \\
    & + \frac{1}{\pi^2} \E \Big[ \sum_{e \neq e' \in \Ec_{\partial \D}} \frac{H_\D(x,a(e)) H_\D(x,b(e))}{H_\D(a(e),b(e))} \frac{H_\D(y,a(e')) H_\D(y,b(e'))}{H_\D(a(e'),b(e'))} \Big]. \notag
\end{align}
We will denote this density by $(x,y) \in \D \times \D \mapsto \corint{L_x(\Pc)L_y(\Pc)}$.
Moreover, for all $x,y \in \D$ with $x \neq y$, we have
\begin{align}\label{E:L_2point2}
\corint{L_x(\Pc)L_y(\Pc)}
= \frac{1}{\pi^2} \lim_{r \to 0} |\log r|^2 \pint(\Pc \cap B(x,r) \neq \varnothing, \Pc \cap B(y,r) \neq \varnothing).
\end{align}
    More generally, for any measurable function $F:\C \times \C \times \Lc \to [0,1]$, the following measure is abolutely continuous w.r.t. Lebesgue measure in $\D \times \D$:
    \begin{equation}
        \label{E:L_2point3}
        \Eint[F(x,y,\Pc) L_x(\Pc) \d x L_y(\Pc) \d y] = \corint{F(x,y,\Pc)L_x(\Pc) L_y(\Pc)} ~\d x \d y,
    \end{equation}
    where the density on the right hand side is defined by this identity.
\end{lemma}

\begin{proof}
    \eqref{E:L_2point2} is a consequence of the fact that the occupation measure of $\Pc$ coincides with its Minkowski content in the gauge $r \mapsto \frac1\pi r^2 |\log r|$; see \eqref{E:Minkowski123} for more details. To prove \eqref{E:L_2point}, let $a,b,a',b' \in \D$ be pairwise distinct boundary points and let $e$ and $e'$ be independent Brownian excursions in $\D$ from $a$ to $b$ and from $a'$ to $b'$ respectively. A small calculation shows that the measures $\E[L_x(e)L_y(e) \d x \d y]$ and $\E[L_x(e)L_y(e') \d x \d y]$ have densities with respect to Lebesgue measure on $\D \times \D$ respectively given by
    \begin{align}\label{E:ghh}
        (x,y) \in \D \times \D \quad \longmapsto \quad & \frac1\pi G_\D(x,y) \frac{H_\D(x,a) H_\D(y,b) + H_\D(x,b) H_\D(y,a)}{H_\D(a,b)} \\
        \text{and} \qquad (x,y) \in \D \times \D \quad \longmapsto \quad & \frac{1}{\pi^2} \frac{H_\D(x,a) H_\D(x,b)}{H_\D(a,b)} \frac{H_\D(y,a') H_\D(y,b')}{H_\D(a',b')}.
    \end{align}
    The first term in \eqref{E:ghh} corresponds to the case where the points are visited by the excursion $e$ in the order $a \to x \to y \to b$ while the second term corresponds to $a \to y \to x \to b$.
    \eqref{E:L_2point} then follows by Fubini. \eqref{E:L_2point3} is a direct generalisation.
    
\end{proof}

\begin{lemma}\label{L:2point_gamma}
    For $\nu_\SLE$-almost every $\gamma$, the following holds. Let $F: \inte(\gamma) \times \inte(\gamma) \times \Lc \to [0,1]$ be a measurable map. The following measure is absolutely continuous w.r.t. Lebesgue measure in $\inte(\gamma)\times \inte(\gamma)$:
    \[
    \int \loopmeasure(\d \Pc \vert \gamma) F(x,y,\Pc) L_x(\Pc) \d x L_y(\Pc) \d y = \langle F(x,y,\Pc) L_x(\Pc) L_y(\Pc) \vert \gamma \rangle^{\rm loop} \d x \d y,
    \]
    where the integral is w.r.t. $\Pc$ and the right hand side is defined by this identity.
    Moreover, if $\varphi_\gamma : \D \to \inte(\gamma)$ is a conformal map, then for all $x, y \in \inte(\gamma)$,
    \[
    \langle F(x, y,\Pc) L_x(\Pc) L_y(\Pc) \vert \gamma \rangle^{\rm loop}
    = \corint{F(x,y, \varphi_\gamma(\Pc)) L_{\varphi_\gamma^{-1}(x)}(\Pc) L_{\varphi_\gamma^{-1}(y)}(\Pc)}.
    \]
\end{lemma}
\begin{proof}
    The proof is similar to the one of Lemma \ref{L:1point_gamma}. We omit the details.
\end{proof}

\subsection{First moment - Proof of \texorpdfstring{\eqref{E:T_level_line_expectation}}{(4.3)}}\label{SS:second_first}

\begin{lemma}\label{L:1point_constant}
    $\corint{L_x(\Pc)}$ does not depend on $x \in \D$. We will denote this common value by $\lambda_0$.
\end{lemma}

\begin{proof}
Let $x,y \in \D$ and let $\varphi : \D \to \D$ be a conformal transformation mapping $x$ to $y$. By Lemma~\ref{L:1point} and conformal invariance of $\Ec_{\partial \D}$ (see Proposition \ref{prop:decomp}) and then by performing a change of variable $e' = \varphi(e)$, we have
\[
\corint{L_y(\Pc)}
= \frac1\pi \E \Big[ \sum_{e' \in \varphi(\Ec_{\partial \D})} \frac{H_\D(y,a(e')) H_\D(y,b(e'))}{H_\D(a(e'),b(e'))} \Big]
= \frac1\pi \E \Big[ \sum_{e \in \Ec_{\partial \D}} \frac{H_\D(y,\varphi(a(e))) H_\D(y,\varphi(b(e)))}{H_\D(\varphi(a(e)),\varphi(b(e)))} \Big].
\]
By conformal covariance of the Poisson kernel (see \eqref{E:Poisson_conformal} and notice that the derivatives of $\varphi$ at the points $a(e)$ and $b(e)$ cancel out), we deduce that
\[
\corint{L_y(\Pc)}
= \frac1\pi \E \Big[ \sum_{e \in \Ec_{\partial \D}} \frac{H_\D(x,a(e)) H_\D(x,b(e))}{H_\D(a(e),b(e))} \Big]
= \corint{L_x(\Pc)}
\]
by Lemma \ref{L:1point}. This concludes the proof.
\end{proof}

\begin{lemma}\label{L:lambda}
$\lambda_0 = 5/\pi$.
\end{lemma}

\begin{proof}
Recall that we denote by $\Gamma$ the space of self-avoiding loops; see Definition~\ref{Def:space_self-avoiding}.
Combining Lemmas \ref{L:1point_gamma} and \ref{L:1point_constant}, we obtain that, for any measurable function $F: \Gamma \to [0,\infty]$, the following measures in $\C$ agree
\begin{equation}\label{E:L_1point_loopmeasure}
\int \d \loopm(\wp) F(\out(\wp)) L_x(\wp) \d x
= \lambda_0 \int \d \nu_\SLE(\gamma) F(\gamma) \indic{x \in \inte(\gamma)} \d x,
\end{equation}
where the integrals are with respect to $\wp$ and $\gamma$ respectively.
This is a key identity. It both involves the constant $\lambda_0$ that we are after and contains only terms that we will be able to compute (or rather estimate asymptotically).

Let us denote by $Q = [0,1]^2$ and let $\eps >0$.
We apply the relation \eqref{E:L_1point_loopmeasure} to $F(\gamma) = \indic{\gamma \subset Q, \diam(\gamma)>\eps}$.  
Since for any loop $\wp$, $\int L_x(\wp) \d x = T(\wp)$, integrating the relation \eqref{E:L_1point_loopmeasure} over $x \in Q$ yields
\begin{equation}
\label{E:pf_lambda1}
    \int \loopm(\d \wp) \indic{\wp \subset Q, \diam(\wp)>\eps} T(\wp)
    = \lambda_0 \int \loopm(\d \wp) \indic{\wp \subset Q, \diam(\wp)>\eps} \area(\inte(\wp)).
\end{equation}
We are going to show that
\begin{equation}\label{E:pf_lambda3}
\lim_{\eps \to 0} \frac{1}{|\log \eps|} \int \loopm(\d \wp) \indic{\wp \subset Q, \diam(\wp)>\eps} T(\wp) = \frac{1}{\pi}
\end{equation}
and
\begin{equation}\label{E:pf_lambda4}
\lim_{\eps \to 0} \frac{1}{|\log \eps|} \int \loopm(\d \wp) \indic{\wp \subset Q, \diam(\wp)>\eps} \area(\inte(\wp)) = \frac{\pi}{5} \times \frac{1}{\pi}.
\end{equation}
Together with \eqref{E:pf_lambda1} this will imply that $\lambda_0 = 5/\pi$ as claimed. As already alluded to, the proof of \eqref{E:pf_lambda4} relies on a result of Garban and Trujillo Ferreras \cite{MR2217292} concerning the expected area of the domain delimited by a Brownian bridge.

Let us start with the proof of \eqref{E:pf_lambda3}.
We are going to show that we can essentially replace the conditions that $\wp \subset Q$ and $\diam(\wp) >\eps$ by the conditions that its root belongs to $Q$ and its duration belongs to $[\eps^2,1]$. By definition of $\loopm$ \eqref{E:loopm}, the resulting integral will be explicitly given by:
\begin{align*}
    \int_Q \d z \int_{\eps^2}^1 \frac{\d t}{t} \frac{1}{2\pi t} \times t = \frac1\pi |\log \eps|
\end{align*}
which is consistent with \eqref{E:pf_lambda3}. We now make this reasonning precise.
By definition of $\loopm$, the left hand side of \eqref{E:pf_lambda1} is equal to
\begin{equation}\label{E:pf_lambda2}
\int_0^\infty \d t \int_Q \d x \frac{1}{2\pi t} \P^{t,x,x}(\wp \subset Q, \diam(\wp)>\eps).
\end{equation}
If $t \geq 1$, we bound
\[
\frac{1}{2\pi t} \P^{t,x,x}(\wp \subset Q, \diam(\wp)>\eps) \leq \frac{1}{2\pi t} \P^{t,x,x}(\wp \subset Q) = p_Q(t,x,x)
\]
where $p_Q$ is the heat kernel in $Q$. Because $t \mapsto \sup_{x \in Q} p_Q(t,x,x)$ decays exponentially fast, this shows that the contribution of the integral from $1$ to $\infty$ is bounded, uniformly in $\eps$. If $t \in [\eps^2 |\log \eps|^{-2},1]$, then we simply bound the probability in \eqref{E:pf_lambda2} by 1, showing that the contribution of the integral for $t \in [\eps^2 |\log \eps|^{-2},1]$ is at most
\[
\int_{\eps^2 |\log \eps|^{-2}}^1 \frac{\d t}{2\pi t} = \frac1\pi (|\log \eps| + \log |\log \eps|).
\]
Finally, if $t \leq \eps^2 |\log \eps|^{-2}$, we bound the probability in \eqref{E:pf_lambda2} by $\P^{t,x,x}(\diam(\wp)>\eps)$. The $x$-projection and $y$-projection of a 2D Brownian bridge are two (independent) 1D Brownian bridges. If the 2D bridge has a diameter larger than $\eps$, then at least one of the two 1D bridges has a diameter larger than $\eps/\sqrt{2}$.
Moreover, by symmetry, the probability that the diameter of a 1D bridge exceeds $\eps/\sqrt{2}$ is at most twice the probability that its maximum exceeds $2^{-3/2} \eps$ (compared to the starting point). By the reflection principle, this latter probability equals $e^{-\eps^2/(4t)}$; see e.g. (3.40) in \cite{MR1121940}.
The contribution of the integral for $t \in [0,\eps^2 |\log \eps|^{-2}]$ is then at most
\[
\int_0^{\eps^2 |\log \eps|^{-2}} \frac{1}{2\pi t} \times 4 e^{-\frac{\eps^2}{4 t}} \d t
= \int_0^{|\log \eps|^{-2}} \frac{2}{\pi t} e^{-\frac{1}{4 t}} \d t
\]
which goes to 0 as $\eps \to 0$. Altogether, we have obtained the upper bound of \eqref{E:pf_lambda3}.
The lower bound can be obtained in a similar manner, showing \eqref{E:pf_lambda3}.

We now move to the proof of \eqref{E:pf_lambda4}. By definition of $\loopm$,
\begin{align}
\nonumber
    & \int \loopm(\d \wp) \indic{\wp \subset Q, \diam(\wp)>\eps} \area(\inte(\wp)) \\
    & = \int_0^\infty \frac{\d t}{t} \int_Q \d x ~p(t,x,x) \E^{x, x,t}[\area(\inte(\wp)) \indic{\wp \subset Q, \diam(\wp)>\eps}] \label{E:pf_lambda6}.
\end{align}
As before, one can argue that the contributions of the integral for $t \geq 1$ and for $t \leq \eps^2 |\log \eps|^{-2}$ are bounded, uniformly in $\eps$. To this end, one can for instance use Cauchy--Schwarz inequality and bound the expectation in \eqref{E:pf_lambda6} by
\[
\E^{t,x,x}[\area(\inte(\wp))^2]^{1/2} \P^{t,x,x}(\wp \subset Q, \diam(\wp) > \eps)^{1/2}.
\]
The area of $\inte(\wp)$ is well concentrated around $t$: $\E^{t,x,x}[\area(\inte(\wp))^2]^{1/2} \leq C t$.
This follows from similar argument as above, using 1D Brownian bridges. We then deal with the probability $\P^{t,x,x}(\wp \subset Q, \diam(\wp) > \eps)$ as in the proof of \eqref{E:pf_lambda3}. Notice in particular that the exponent 1/2 coming from Cauchy--Schwarz inequality does not affect the reasoning.
For $t \in [\eps^2 |\log \eps|^{-2},1]$, we bound the expectation in \eqref{E:pf_lambda6} by $\E^{x \to x,t}[\area(\inte(\wp))]$. By Brownian scaling, this is equal to $t$ times the area of the interior of a Brownian bridge of duration 1 which is equal to $\pi/5$; see \cite{MR2217292}.
We deduce that the integral for $t \in [\eps^2 |\log \eps|^{-2},1]$ is at most
\[
\frac{\pi}{5} \int_{\eps^2 |\log \eps|^{-2}}^1 \frac{\d t}{2\pi t} = \frac{\pi}{5} \frac{1}{\pi} (|\log \eps| + \log |\log \eps|).
\]
Putting things together leads to the upper bound of \eqref{E:pf_lambda4}. The lower bound is similar, concluding the proof.
\end{proof}

We can finish this section with the proof we were after:

\begin{proof}[Proof of Theorem \ref{T:level_line}, \eqref{E:T_level_line_expectation}]
    It is a consequence of Lemma \ref{L:1point}, Lemma~\ref{L:1point_constant} and Lemma \ref{L:lambda}.
\end{proof}

\subsection{Second moment - Proof of \texorpdfstring{\eqref{E:T_level_line_bc}}{(4.4)}}\label{SS:second_moment}

This section gathers the second moment computations needed in the proof of \eqref{E:T_level_line_bc}. It is the most technically challenging part of the article. This section is structured in a way as to reflect the main steps of the proof.

Recall that the two-point function $\corint{L_x(\Pc) L_y(\Pc)}$ is defined in Lemma \ref{L:2point}.

\begin{proposition}\label{P:upper_bound}
    There exists $C>0$ such that for all $x,y \in \D$,
    \[
    \corint{L_x(\Pc)L_y(\Pc)} \leq C \max(1, G_\D(x,y)).
    \]
\end{proposition}

\begin{proof}
    It follows from Lemma \ref{L:2point}, Corollary \ref{C:upper_bound_2point} and conformal invariance.
\end{proof}

In the proof of \eqref{E:T_level_line_bc}, we will make use of the following notational convention:

\begin{notation}\label{N:arc}
    If $z, w \in \partial \D$ are two points of the unit circle, we will denote by $(z,w)$ (resp. $[z,w]$) the counterclockwise boundary arc from $z$ to $w$, excluding the points $z$ and $w$ (resp. including the points $z$ and $w$).
\end{notation}

\begin{proof}[Proof of Theorem \ref{T:level_line}, \eqref{E:T_level_line_bc}]
Let $(f_\eps)_\eps$ be a sequence of test functions satisfying Assumption \ref{assumption_feps} and denote by
\[
I_\eps = \int_\D f_\eps(x) L_x(\Pc) \d x.
\]
We want to show that $I_\eps \to 5/\pi$ in L$^1(\pint)$. As before, we will denote by $\Ec_{\partial \D}$ the associated collection of excursions in $\D$ with endpoints on $\partial \D$. We start by introducing a ``good'' event.

\medskip

\noindent\textbf{Good event $G_\delta(z)$.}
Let $\delta >0$ and $z \in \partial \D$.
Let
\begin{equation}
    \label{E:z-z+}
    z_- = z_-(\delta) = e^{-i \delta/2} z
    \quad \text{and} \quad
    z_+ = z_+(\delta) = e^{i \delta/2} z.
\end{equation}
Recalling Notation \ref{N:arc}, we define the following event
\begin{align}
    \label{E:def_good}
    G_\delta(z) & := \{ \forall e \in \Ec_{\partial \D}, a(e) \notin (z_-,z_+) \text{ or } b(e) \notin (z_-,z_+) \implies e \cap B(z,\delta^{100}) = \varnothing \} \\
    & \hspace{20pt} \cap \{ \forall e \in \Ec_{\partial \D} \text{ with } a(e) \in (z_-,z_+) \text{ or } b(e) \in (z_-,z_+), \diam(e) < \delta^{1/10} \} \notag
\end{align}
and the slightly more restrictive version
\begin{align}
    \label{E:def_good'}
    G_\delta'(z) & := \{ \forall e \in \Ec_{\partial \D}, a(e) \notin (e^{-i\delta/4}z,e^{i\delta/4} z) \text{ or } b(e) \notin (e^{-i\delta/4}z,e^{i\delta/4} z) \implies e \cap B(z,2\delta^{100}) = \varnothing \} \\
    & \hspace{20pt} \cap \{ \forall e \in \Ec_{\partial \D} \text{ with } a(e) \in (e^{-i \delta}z,e^{i \delta}z) \text{ or } b(e) \in (e^{-i \delta}z,e^{i \delta}z), \diam(e) < \delta^{1/10}/2 \}. \notag
\end{align}
See Figure \ref{Figure:good_event} for an illustration of the event $G_\delta(z)$.
Denote by
\[
I_\eps^{\mathrm{g}} := \int_\D f_\eps(x) L_x(\Pc) \mathbf{1}_{G_\delta'(x/|x|)} \d x
\]
where the superscript ``g'' stands for ``good''. $I_\eps^{\mathrm{g}}$ implicitly depends on $\delta$.

\begin{figure}
   \centering
   \begin{subfigure}[t]{.3\columnwidth}
    \def\svgwidth{\columnwidth}
\begingroup%
  \makeatletter%
  \providecommand\color[2][]{%
    \errmessage{(Inkscape) Color is used for the text in Inkscape, but the package 'color.sty' is not loaded}%
    \renewcommand\color[2][]{}%
  }%
  \providecommand\transparent[1]{%
    \errmessage{(Inkscape) Transparency is used (non-zero) for the text in Inkscape, but the package 'transparent.sty' is not loaded}%
    \renewcommand\transparent[1]{}%
  }%
  \providecommand\rotatebox[2]{#2}%
  \newcommand*\fsize{\dimexpr\f@size pt\relax}%
  \newcommand*\lineheight[1]{\fontsize{\fsize}{#1\fsize}\selectfont}%
  \ifx\svgwidth\undefined%
    \setlength{\unitlength}{176.7386126bp}%
    \ifx\svgscale\undefined%
      \relax%
    \else%
      \setlength{\unitlength}{\unitlength * \real{\svgscale}}%
    \fi%
  \else%
    \setlength{\unitlength}{\svgwidth}%
  \fi%
  \global\let\svgwidth\undefined%
  \global\let\svgscale\undefined%
  \makeatother%
  \begin{picture}(1,1.06156812)%
    \lineheight{1}%
    \setlength\tabcolsep{0pt}%
    \put(0,0){\includegraphics[width=\unitlength,page=1]{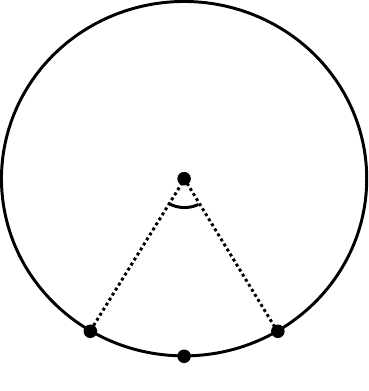}}%
    \put(0.47356657,0.43190962){\color[rgb]{0,0,0}\makebox(0,0)[lt]{\lineheight{1.25}\smash{\begin{tabular}[t]{l}$\delta$\end{tabular}}}}%
    \put(0.77301846,0.11852592){\color[rgb]{0,0,0}\makebox(0,0)[lt]{\lineheight{1.25}\smash{\begin{tabular}[t]{l}$z_+$\end{tabular}}}}%
    \put(0.18772616,0.10021721){\color[rgb]{0,0,0}\makebox(0,0)[lt]{\lineheight{1.25}\smash{\begin{tabular}[t]{l}$z_-$\end{tabular}}}}%
    \put(0.45301602,0.03034512){\color[rgb]{0,0,0}\makebox(0,0)[lt]{\lineheight{1.25}\smash{\begin{tabular}[t]{l}$z$\end{tabular}}}}%
    \put(0.23683414,0.91540964){\color[rgb]{0,0,0}\makebox(0,0)[lt]{\lineheight{1.25}\smash{\begin{tabular}[t]{l}$\D$\end{tabular}}}}%
    \put(0,0){\includegraphics[width=\unitlength,page=2]{fig1.pdf}}%
    \put(0.53171558,0.00488377){\color[rgb]{0,0,0}\makebox(0,0)[lt]{\lineheight{1.25}\smash{\begin{tabular}[t]{l}$B(z,\delta^{100})$\end{tabular}}}}%
    \put(0,0){\includegraphics[width=\unitlength,page=3]{fig1.pdf}}%
  \end{picture}%
\endgroup%

   \end{subfigure}
   \hspace{40pt}
   \raisebox{5pt}{
   \begin{subfigure}[t]{.3\columnwidth}
    \def\svgwidth{\columnwidth}
\begingroup%
  \makeatletter%
  \providecommand\color[2][]{%
    \errmessage{(Inkscape) Color is used for the text in Inkscape, but the package 'color.sty' is not loaded}%
    \renewcommand\color[2][]{}%
  }%
  \providecommand\transparent[1]{%
    \errmessage{(Inkscape) Transparency is used (non-zero) for the text in Inkscape, but the package 'transparent.sty' is not loaded}%
    \renewcommand\transparent[1]{}%
  }%
  \providecommand\rotatebox[2]{#2}%
  \newcommand*\fsize{\dimexpr\f@size pt\relax}%
  \newcommand*\lineheight[1]{\fontsize{\fsize}{#1\fsize}\selectfont}%
  \ifx\svgwidth\undefined%
    \setlength{\unitlength}{176.73862678bp}%
    \ifx\svgscale\undefined%
      \relax%
    \else%
      \setlength{\unitlength}{\unitlength * \real{\svgscale}}%
    \fi%
  \else%
    \setlength{\unitlength}{\svgwidth}%
  \fi%
  \global\let\svgwidth\undefined%
  \global\let\svgscale\undefined%
  \makeatother%
  \begin{picture}(1,1.03554316)%
    \lineheight{1}%
    \setlength\tabcolsep{0pt}%
    \put(0,0){\includegraphics[width=\unitlength,page=1]{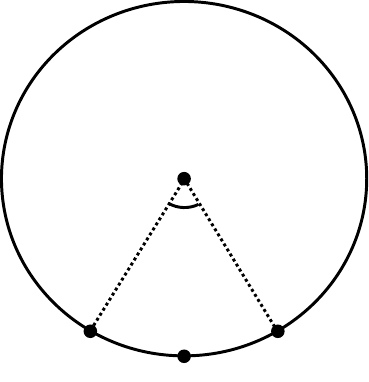}}%
    \put(0.47356654,0.40588468){\color[rgb]{0,0,0}\makebox(0,0)[lt]{\lineheight{1.25}\smash{\begin{tabular}[t]{l}$\delta$\end{tabular}}}}%
    \put(0.77301853,0.09250105){\color[rgb]{0,0,0}\makebox(0,0)[lt]{\lineheight{1.25}\smash{\begin{tabular}[t]{l}$z_+$\end{tabular}}}}%
    \put(0.18772621,0.07419231){\color[rgb]{0,0,0}\makebox(0,0)[lt]{\lineheight{1.25}\smash{\begin{tabular}[t]{l}$z_-$\end{tabular}}}}%
    \put(0.45301608,0.00432017){\color[rgb]{0,0,0}\makebox(0,0)[lt]{\lineheight{1.25}\smash{\begin{tabular}[t]{l}$z$\end{tabular}}}}%
    \put(0.23683425,0.88938478){\color[rgb]{0,0,0}\makebox(0,0)[lt]{\lineheight{1.25}\smash{\begin{tabular}[t]{l}$\D$\end{tabular}}}}%
    \put(0,0){\includegraphics[width=\unitlength,page=2]{fig2.pdf}}%
    \put(0.26875482,0.41691828){\color[rgb]{1,0,0}\makebox(0,0)[lt]{\lineheight{1.25}\smash{\begin{tabular}[t]{l}$\delta^{\frac{1}{100}}$\end{tabular}}}}%
  \end{picture}%
\endgroup%

   \end{subfigure}
   }
\caption{Illustration of the two scenarios forbidden by the good event $G_\delta(z)$. On the left, a blue loop with one end point outside of $(z_-,z_+)$ hits the ball $B(z,\delta^{100})$ depicted in yellow. On the right, a blue loop with one endpoint in $(z_-,z_+)$ has a diameter exceeding $\delta^{1/100}$.}\label{Figure:good_event}
\end{figure}

We now state two key intermediate lemmas. We will then show that Theorem \ref{T:level_line} follows from these two results (this step will be elementary) before returning to their proofs.

\begin{lemma}\label{L:first_good}
The random variable $I_\eps - I_\eps^{\mathrm{g}}$ is small in L$^1$ in the sense that:
\begin{equation}
    \label{E:L_first_good}
\lim_{\delta \to 0} \limsup_{\eps \to 0} \Eint[I_\eps - I_\eps^{\mathrm{g}}] = 0.
\end{equation}
\end{lemma}

\begin{lemma}\label{L:second_good}
For all $\delta>0$ fixed,
\[
\limsup_{\eps \to 0} \Eint \Big[ \int_{\D \times \D} L_x(\Pc)L_y(\Pc) \mathbf{1}_{G_\delta'(x/|x|) \cap G_\delta'(y/|y|)} f_\eps(x) f_\eps(y) \indic{|x-y|>\delta^{1/100}} \d x \d y \Big] \leq \Big( \frac{5}{\pi} \Big)^2.
\]
\end{lemma}

We assume that Lemmas \ref{L:first_good} and \ref{L:second_good} hold and conclude the proof of Theorem \ref{T:level_line}.
Let $\delta >0$. By the triangle inequality and Cauchy--Schwarz, we have
\[
\Eint [|I_\eps - 5/\pi|] \leq 
\Eint [|I_\eps^{\mathrm{g}} - 5/\pi|^2]^{1/2} + \Eint[I_\eps - I_\eps^{\mathrm{g}}].
\]
The second right hand side term is handled by Lemma \ref{L:first_good}. For the first term, we expand the square and, recalling that $\E I_\eps = 5/\pi$ (Theorem \ref{T:level_line}, \eqref{E:T_level_line_expectation}), we get that
\begin{align*}
& \Eint [|I_\eps^{\mathrm{g}} - 5/\pi|^2]
= \Eint[(I_\eps^{\mathrm{g}})^2] - (5/\pi)^2 + (10/\pi) \Eint[ I_\eps - I_\eps^{\mathrm{g}}].
\end{align*}
By subadditivity of $u \mapsto \sqrt{u}$, this shows that
\[
\Eint [|I_\eps - 5/\pi|] \leq 
(\Eint [(I_\eps^{\mathrm{g}})^2] - (5/\pi)^2)_+^{1/2} + (10/\pi)^{1/2} \Eint[I_\eps - I_\eps^{\mathrm{g}}]^{1/2} + \Eint[I_\eps - I_\eps^{\mathrm{g}}],
\]
where $(u)_+ = \max(u,0)$.
We further expand
\begin{align*}
\Eint[(I_\eps^{\mathrm{g}})^2] & = \Eint\Big[ \int_{\D \times \D} L_x(\Pc)L_y(\Pc) \mathbf{1}_{G_\delta'(x/|x|) \cap G_\delta'(y/|y|)} f_\eps(x) f_\eps(y) \indic{|x-y|>\delta^{1/100}} \d x \d y \Big] \\
& + \Eint\Big[ \int_{\D \times \D} L_x(\Pc)L_y(\Pc) \mathbf{1}_{G_\delta'(x/|x|) \cap G_\delta'(y/|y|)} f_\eps(x) f_\eps(y) \indic{|x-y|\leq\delta^{1/100}} \d x \d y \Big]. 
\end{align*}
By Lemma \ref{L:second_good}, the limsup as $\eps \to 0$ of the first right hand side term is at most $(5/\pi)^2$. Putting things together, this shows that
\begin{align*}
& \limsup_{\eps \to 0} \Eint [|I_\eps - 5/\pi|] \leq 
\limsup_{\eps \to 0} ((10/\pi)^{1/2} \Eint[I_\eps - I_\eps^{\mathrm{g}}]^{1/2} + \Eint[I_\eps - I_\eps^{\mathrm{g}}])\\
& \hspace{60pt} + \limsup_{\eps \to 0} \Big( \int_{\D \times \D} \corint{L_x(\Pc)L_y(\Pc)} f_\eps(x) f_\eps(y) \indic{|x-y|\leq\delta^{1/100}} \d x \d y \Big)^{1/2}.
\end{align*}
The left hand side is independent of $\delta$. On the other hand, the first right hand side limsup vanishes as $\delta \to 0$ by Lemma \ref{L:first_good}, whereas the second right hand side limsup vanishes by Proposition \ref{P:upper_bound} and the assumption \eqref{E:assumption_feps2} on $(f_\eps)_\eps$. This concludes the proof that $I_\eps \to 5/\pi$ in L$^1(\pint)$.
To finish the proof of Theorem \ref{T:level_line}, it remains to prove Lemmas \ref{L:first_good} and \ref{L:second_good}.

\begin{proof}[Proof of Lemma \ref{L:first_good}]
Let $0<\eps<\delta$. The ``correlation functions'' written below are justified by Lemma \ref{L:1point}. We have
\begin{align*}
    \Eint[I_\eps - I_\eps^{\mathrm{g}}]
    = \int _\D f_\eps(x) \corint{L_x(\Pc) \mathbf{1}_{G_\delta'(x/|x|)^c}} \d x.
\end{align*}
We are going to show that
\begin{equation}
    \label{E:pf_good_goal}
    \sup_{\substack{x \in \D\\|x|>1-\delta}} \corint{L_x(\Pc) \mathbf{1}_{G_\delta'(x/|x|)^c}} \to 0
    \quad \text{as} \quad \delta \to 0.
\end{equation}
Since $f_\eps(x)$ vanishes for $|x| \leq 1-\eps$ and since $\int f_\eps =1$, this will prove \eqref{E:L_first_good}.
Let $x \in \D$ with $|x|>1-\delta$. Let $G_1$ and $G_2$ be the first and second events appearing on the right hand side of \eqref{E:def_good'} with $z = x/|x|$. We will show separately that
\begin{equation}
    \label{E:pf_good_goal2}
    \corint{L_x(\Pc) \mathbf{1}_{G_1^c}} \to 0
    \quad \text{as} \quad \delta \to 0,
\end{equation}
and
\begin{equation}
    \label{E:pf_good_goal3}
    \corint{L_x(\Pc) \mathbf{1}_{G_2^c}} \to 0
    \quad \text{as} \quad \delta \to 0,
\end{equation}
uniformly in $|x| >1-\delta.$
We start with \eqref{E:pf_good_goal3}.

\begin{figure}
   \centering
   \begin{subfigure}{.8\columnwidth}
    \def\svgwidth{\columnwidth}
   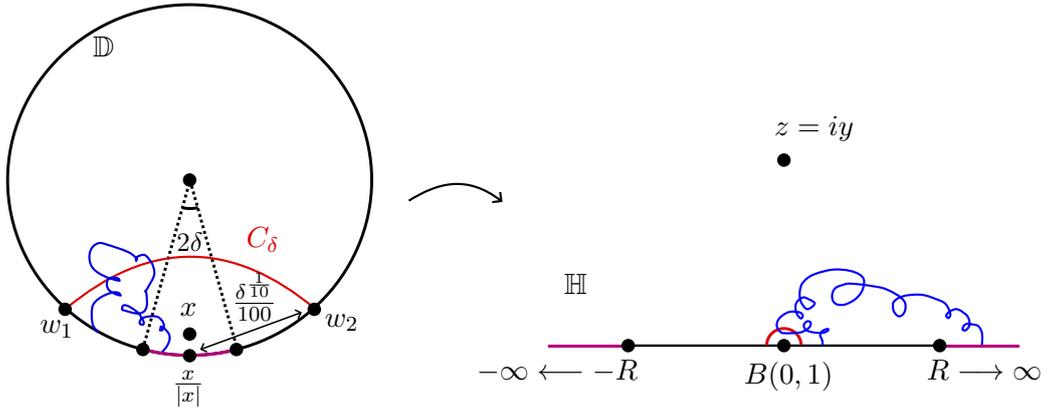
   \end{subfigure}
\caption{Illustration of notations and an event appearing in the proof of Lemma \ref{L:first_good}. On the left, an excursion with an endpoint in $(e^{-i\delta} x/|x|,e^{+i\delta} x/|x|)$ has a diameter at least $\delta^{1/10}/2$ and thus intersects $C_\delta$. The right picture is the image of the left picture under the conformal map sending $x/|x|$ to infinity, $w_1$ to $1$ and $-x/|x|$ to $0$. The points $e^{\pm i\delta} x/|x|$ are then mapped to $\mp R$, $w_2$ to $-1$ and $x$ to $z$.}\label{fig:first_moment}
\end{figure}

Consider the two boundary points $w_1, w_2$ that are at distance $\delta^{1/10}/100$ to $x/|x|$ and let $C_\delta$ be the hyperbolic geodesic between these two points; see Figure \ref{fig:first_moment}.
If an excursion $e$ with at least one endpoint in $(e^{-i\delta} x/|x|,e^{+i\delta} x/|x|)$ has a diameter at least $\delta^{1/10}/2$, it has to intersect $C_\delta$. Hence,
\begin{align}
    \notag
\corint{L_x(\Pc) \mathbf{1}_{G_2^c}}
& \leq \corint{L_x(\Pc) \indic{\exists e \in \Ec_{\partial \D}, a(e) \text{ or } b(e) \in (e^{-i\delta} x/|x|,e^{+i\delta} x/|x|), e \cap C_\delta \neq \varnothing}} \\
\label{E:pf_good1}
& \leq 2 \corint{L_x(\Pc) \indic{\exists e \in \Ec_{\partial \D}, a(e) \in (e^{-i\delta} x/|x|,e^{+i\delta} x/|x|), e \cap C_\delta \neq \varnothing}}.
\end{align}
We now map the unit disc $\D$ to the upper half plane $\Hb$, send $x/|x|$ to $\infty$, $-x/|x|$ to 0 and $w_1$ to 1. By symmetry $w_2$ is sent to $-1$ and $C_\delta$ to the hyperbolic geodesic between 1 and $-1$ which is simply $\partial B(0,1) \cap \Hb$. Moreover, the arc $(e^{-i\delta} x/|x|,e^{+i\delta} x/|x|)$ is sent to $(-\infty, -R) \cup (R, \infty)$ where $R = R(\delta) \to +\infty$ as $\delta \to 0$. We will denote by $z$ the image of $x$. By symmetry, $z$ is purely imaginary.
See Figure \ref{fig:first_moment} for a schematic representation of these notations.
Since $|x| > 1-\delta$, $\Im(z) \to \infty$ as $\delta \to 0.$ By conformal invariance and then by a union bound, the right hand side of \eqref{E:pf_good1} is equal to
\[
2\langle L_z(\Pc) \indic{\exists e \in \Ec_\R, |a(e)|>R, e \cap B(0,1) \neq \varnothing} \rangle^\wired_{\Hb,\R}
\leq 2 \Big\langle \sum_{\substack{e,e' \in \Ec_\R\\|a(e')|>R}} L_z(e) \indic{e' \cap B(0,1) \neq \varnothing} \Big\rangle^\wired_{\Hb,\R}.
\]
The right hand side term is further equal to
\begin{align}
    \label{E:pf_good3}
    & \frac{2}{\pi} \int_{\partial B(0,1) \cap \Hb} \E \Big[ \sum_{\substack{e \neq e' \in \Ec_\R \\ |a(e')|>R}} \frac{H_\Hb(z,a(e)) H_\Hb(z,b(e))}{H_\Hb(a(e),b(e))} \frac{H_{\Hb \setminus B(0,1)}(z',a(e')) H_\Hb(z',b(e'))}{H(a(e'),b(e'))} \Big] \d z' \\
    \label{E:pf_good4}
    & + \frac{2}{\pi} \int_{\partial B(0,1) \cap \Hb} H_{\Hb \setminus B(0,1)}(z,z') \E \Big[ \sum_{\substack{e \in \Ec_\R \\ |a(e)|>R}} \frac{H_{\Hb \setminus B(0,1)}(z,a(e)) H_\Hb(z',b(e))}{H_\Hb(a(e),b(e))} \Big] \d z' \\
    \label{E:pf_good5}
    & + 2 \int_{\partial B(0,1) \cap \Hb} G_{\Hb}(z,z') \E \Big[ \sum_{\substack{e \in \Ec_\R \\ |a(e)|>R}} \frac{H_{\Hb \setminus B(0,1)}(z',a(e)) H_\Hb(z,b(e))}{H_\Hb(a(e),b(e))} \Big] \d z'.
\end{align}
Instead of writing long formulas, we explain in words where these expressions come from. 
In each of these terms, the point $z'$ corresponds to the first hitting point of $B(0,1)$ by the excursion $e'$ when viewed as oriented from $a(e')$ to $b(e')$.
The term \eqref{E:pf_good3} corresponds to the case $e \neq e'$ whereas the terms \eqref{E:pf_good4} and \eqref{E:pf_good5} correspond to the case $e = e'$. In \eqref{E:pf_good4} (resp. \eqref{E:pf_good5}), the excursion $e$ visits the points in the following order: $a(e) \to z \to z' \to b(e)$ (resp. $a(e) \to z' \to z \to b(e)$). In \eqref{E:pf_good5} for instance, conditionally on $a(e)$ and $b(e)$, the infinitesimal probability that an excursion in $\Hb$ from $a(e)$ to $b(e)$ hits for the first time $B(0,1)$ at $z'$ and then hits $z$ before reaching $b(e)$ is given by
\[
\frac{H_{\Hb \setminus B(0,1)}(z',a(e)) G_{\Hb}(z,z') H_\Hb(z,b(e))}{H_\Hb(a(e),b(e))}.
\]
The other terms are similar.

We start by bounding the term in \eqref{E:pf_good3}.
Let $z' = x' + iy' \in \partial B(0,1) \cap \Hb$. Let us denote by $y = \Im(z)$ and recall that $z$ is purely imaginary.
For $n \geq 0$, let 
\[
I_{-1}' = [x'-y',x'+y'] \quad \text{and} \quad
I_n' = [x'-2^{n+1}y',x'-2^ny') \cup (x'+2^ny',x'+2^{n+1}y'].
\]
Define similarly $I_n$, $n \geq -1$, with $0$ and $y$ instead of $x'$ and $y'$. 
Let $n, m, n', m' \geq 1$.
When $a(e) \in I_n$, $b(e) \in I_m$, $a(e') \in I'_{n'}$ and $b(e') \in I'_{m'}$, we can use the explicit expressions \eqref{E:Poisson} and \eqref{E:Poisson_boundary} of the Poisson kernels to bound
\[
\frac{H_\Hb(z,a(e)) H_\Hb(z,b(e))}{H_\Hb(a(e),b(e))}
\leq C y^{-2} 2^{-2n-2m} (a(e) - b(e))^2
\]
and
\[
\frac{H_{\Hb \setminus B(0,1)}(z',a(e')) H_\Hb(z',b(e'))}{H(a(e'),b(e'))}
\leq C {y'}^{-2} 2^{-2n'-2m'} (a(e') - b(e'))^2.
\]
We deduce that the expectation in \eqref{E:pf_good3} is at most
\begin{align}\label{E:pf_good2}
   C y^{-2} {y'}^{-2} \sum 2^{-2(n+m+n'+m')} \E \Big[ \sum_{\substack{e \neq e' \in \Ec_\R \\ a(e) \in I_n, b(e) \in I_m \\ a(e') \in I'_{n'}, b(e') \in I'_{m'}}} (a(e) - b(e))^2 (a(e')-b(e'))^2 \Big]
\end{align}
where the first sum is over $n,m,n',m' \geq -1$ such that $x'+2^{n'+1}y' \geq R$. We can bound the above expectation by $\sqrt{X X'}$ where
\[
X = \E \Big[ \Big( \sum_{\substack{e \in \Ec_\R \\ a(e) \in I_n, b(e) \in I_m}} (a(e) - b(e))^2 \Big)^2 \Big]
\]
and similarly for $X'$. By \eqref{E:C_interval_excursion2}, $X \leq C y^4 2^{3 m \vee n + m \wedge n} \leq C y^4 2^{3m+3n}$ and $X' \leq C {y'}^4 2^{3m'+3n'}$. \eqref{E:pf_good2} is thus at most
\[
C \sum_{n' \geq \floor{\log_2(R/y')}} 2^{-n'/2} \leq C \sqrt{y'} R^{-1/2}.
\]
Integrating with respect to $z'=x'+iy'$ then shows that the whole term written in \eqref{E:pf_good3} is at most $C R^{-1/2}$.

We now bound the term in \eqref{E:pf_good4}. We first get an upper bound by removing the condition that $|a(e)|>R$. As we will see, doing so will still produce a good upper bound because $\Im(z) \to \infty$. We can then use similar arguments as before to show that the expectation in \eqref{E:pf_good4} is bounded by a universal constant. The term \eqref{E:pf_good4} is thus at most
\begin{align*}
    C \int_{\partial B(0,1) \cap \Hb} H_{\Hb \setminus B(0,1)}(z,z') \d z'.
\end{align*}
This integral corresponds to the probability that a Brownian motion starting at $z=iy$ hits $B(0,1)$ before $\R$. By considering an appropriate conformal map, we can see that this probability is of the same order as the probability of hitting $[-1,1]$ before $\R \setminus [-1,1]$ which is equal to $\frac{2}{\pi} \arctan(1/y)$. Since $y \to \infty$ as $\delta \to 0$ (uniformly in $x \in \D$ with $|x|>1-\delta$), this proves that \eqref{E:pf_good4} goes to zero as $\delta \to 0$.
Similarly, the term \eqref{E:pf_good5} is bounded by
\[
C \int_{\partial B(0,1) \cap \Hb} G_{\Hb}(z,z') \d z'
\]
which goes to zero as $\delta \to 0$. Altogether, we have proved that each term in \eqref{E:pf_good3}, \eqref{E:pf_good4} and \eqref{E:pf_good5} vanishes as $\delta \to 0$, uniformly in $|x|>1-\delta$. This concludes the proof of \eqref{E:pf_good_goal3}. The proof of \eqref{E:pf_good_goal2} is similar. This finishes the proof of the lemma.
\end{proof}

\begin{proof}[Proof of Lemma \ref{L:second_good}]
In this proof, we will write correlation functions that are justified by Lemma~\ref{L:2point}.
Let $z, w \in \partial \D$ be such that $|z-w|>\delta^{1/100}$. Let $(z_\eps)_\eps$ and $(w_\eps)_\eps$ be two (deterministic) sequences of points in $\D$ with $|z-z_\eps| < \eps$ and $|w-w_\eps|<\eps$. By Proposition~\ref{P:upper_bound} and dominated convergence theorem, it is enough to show that
\[
    \limsup_{\eps \to 0} \corint{L_{z_\eps}(\Pc)L_{w_\eps}(\Pc) \mathbf{1}_{G_\delta'(z_\eps/|z_\eps|) \cap G_\delta'(w_\eps/|w_\eps|)}} \leq (5/\pi)^2.
\]
In this proof, we will always assume that $\eps < \delta^{1000}$. For such small values of $\eps$, we have the inclusion $G_\delta'(z_\eps/|z_\eps|) \subset G_\delta(z)$ where we recall that $G_\delta(z)$ and $G_\delta'(z_\eps/|z_\eps|)$ are defined in \eqref{E:def_good} and \eqref{E:def_good'} respectively. We are thus aiming to show that
\begin{equation}
    \label{E:pf_goal2}
    \limsup_{\eps \to 0} \corint{L_{z_\eps}(\Pc)L_{w_\eps}(\Pc) \mathbf{1}_{G_\delta(z) \cap G_\delta(w)}} \leq (5/\pi)^2.
\end{equation}

\begin{figure}
   \centering
   \begin{subfigure}{.4\columnwidth}
    \def\svgwidth{\columnwidth}
   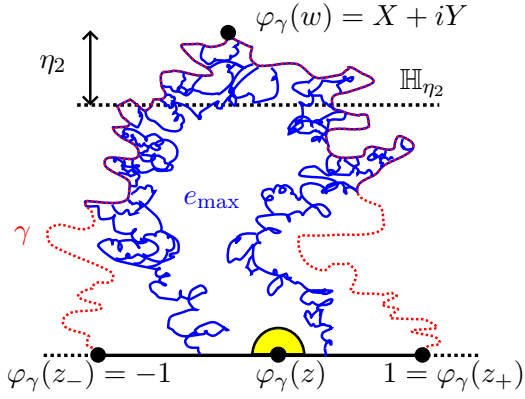
   \end{subfigure}
\caption{Illustration of the setup of the proof of Lemma \ref{L:second_good}. The outer boundary $\gamma$ is the doted red curve and $e_{\max}$ is the blue excursion. On the event $E_1(\Ec_{[-1,1]})$, $e_{\max}$ is the only excursion reaching $\Hb_{\eta_2}$. On the event $E_2(\Ec_{[-1,1]})$, $e_{\max}$ does not visit $\varphi_\gamma(B(z,\delta^{100}))$ which is the region depicted in yellow.}\label{fig:lemma2nd}
\end{figure}

\noindent \textbf{Brownian loop in $\Hb$ wired on $[-1,1]$ -- Setup.}
We start by introducing the setup. See Figure~\ref{fig:lemma2nd} for an illustration.
Consider a Brownian loop $\Pc \sim \P_{\Hb,[-1,1]}^\wired$ in $\Hb$ wired on $[-1,1]$.
$\Pc$ is the concatenation of excursions $e \in \Ec_{[-1,1]}$ in $\Hb$ attached to $[-1,1]^2$.
    Let $\gamma = \out(\Pc)$. Recalling the definition \eqref{E:z-z+} of $z_-$ and $z_+$, let $\varphi_\gamma : \D \to \inte(\gamma)$ be the unique conformal map that sends $z_-$ to $-1$, $z_+$ to $1$ and $w$ to the boundary point with maximal imaginary part (which is a.s. unique). We will denote by $X = \Re(\varphi_\gamma(w))$ and $Y = \Im(\varphi_\gamma(w))$.
    The compact $\overline{\bigcup_{e \in \Ec_{[-1,1]}} \varphi_\gamma^{-1}(e)}$ is the closure of the union of excursions in $\D$ with both endpoints on $\partial \D$ whose law agree with $\pint$. With a slight abuse of notation, we will denote this collection of excursions by $\varphi_\gamma^{-1}(\Ec_{[-1,1]})$. In addition, we will denote by $\varphi_\gamma(G_\delta(z) \cap G_\delta(w))$ the event that $G_\delta(z) \cap G_\delta(w)$ occurs for the collection of excursions $\varphi_\gamma^{-1}(\Ec_{[-1,1]})$.
    Similarly to Lemma \ref{L:2point_gamma}, conformal invariance implies that
    \begin{equation}\label{E:important}
    \E_{\Hb,[-1,1]}^\wired[ L_{\varphi_\gamma(z_\eps)}(\Pc) L_{\varphi_\gamma(w_\eps)}(\Pc) \mathbf{1}_{\varphi_\gamma(G_\delta(z) \cap G_\delta(w))} \vert \gamma ] = \corint{L_{z_\eps}(\Pc)L_{w_\eps}(\Pc) \mathbf{1}_{G_\delta(z) \cap G_\delta(w)}} \qquad \text{a.s.}
    \end{equation}
    The left hand side is a correlation function that is analogous to the one on the right hand side for the law $\P_{\Hb,[-1,1]}^\wired(\cdot \vert \gamma)$. To ease some notations below, we keep writing it in this way in this proof instead of using the notation of correlation functions.

    Almost surely, there is a unique excursion $e_{\max} \in \Ec_{[-1,1]}$ with maximal imaginary part, i.e. such that $\max \Im(e_{\max}) > \max \Im(e)$ for all $e \in \Ec_{[-1,1]} \setminus \{e_{\max}\}$. This excursion will play a special role in the following.
    
    Let $\eta_1,\eta_2 >0$ be small with $\eta_1 \ll \eta_2 \ll \delta$. In the following, we will always assume that $\eps$ is smaller than $\eta_1$. We consider the following half plane and horizontal strips:
    \[
    \Hb_{\eta_2} = \Hb + i (Y - \eta_2),
    \qquad
    S_{\eta_2} = \{ z \in \C: 0< \Im(z) < Y - \eta_2 \},
    \]
    \[
    S_{\max} = \{ z \in \C: 0< \Im(z) < Y \}.
    \]
    We will consider four events, two of them concerning solely the outer boundary $\gamma$ and two other events concerning more globally the set of excursions $\Ec_{[-1,1]}$. We define the events:
\begin{itemize}
    \item $E_1(\gamma)$: $\inte(\gamma) \cap \Hb_{\eta_2}$ is contained in $\varphi_\gamma(B(w,\delta^{100}))$ and contains $\varphi_\gamma(B(w,\eta_1))$;
    \item $E_2(\gamma)$: $\gamma$ does not intersect $\partial \Hb_{\eta_2} \cap ((-\infty, X-\sqrt{\eta_2}) \cup (X+\sqrt{\eta_2},\infty))$;
    \item $E_1(\Ec_{[-1,1]})$: $e_{\max}$ is the only excursion of $\Ec_{[-1,1]}$ that reaches $\Hb_{\eta_2}$;
    \item $E_2(\Ec_{[-1,1]})$: $e_{\max}$ does not intersect $\varphi_\gamma(B(z,\delta^{100}))$.
\end{itemize}
Note that $E_1(\gamma) \cap E_2(\gamma)$ has a positive probability. With some effort, one should be able to show that $\P(E_1(\gamma) \cap E_2(\gamma)) \to 1$ as $\eta_1 \to 0$ and then $\eta_2\to 0$, but we will not need this stronger fact.

\begin{figure}
   \centering
   \begin{subfigure}{.3\columnwidth}
    \def\svgwidth{\columnwidth}
   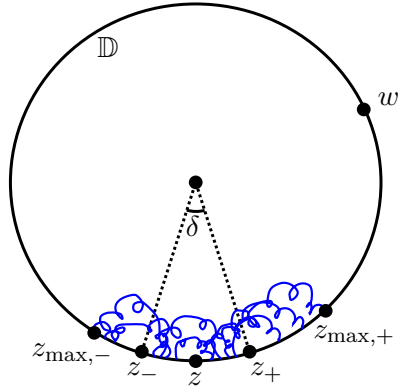
   \end{subfigure}
\caption{Schematic representation of all the excursions in $\varphi_\gamma^{-1}(\Ec_{[-1,1]})$ with at least one endpoint in $(z_-,z_+)$. $z_{\max,+}$ (resp. $z_{\max,-}$) is the endpoint of such an excursion that lies on the arc $(z,w)$ (resp. $(w,z)$) and whose distance to $z$ is maximal. By definition, there is no excursion in $\varphi_\gamma^{-1}(\Ec_{[-1,1]})$ with one endpoint on $(z_-,z_+)$ and one endpoint on $(z_{\max,+},z_{\max,-})$.}\label{fig:zmax+}
\end{figure}

    \medskip

    \noindent \textbf{Inclusion of events.}
    In this paragraph, we will show that
    \begin{equation}
        \label{E:pf_good_inclusion1}
        \varphi_\gamma(G_\delta(z)) \subset E_2(\Ec_{[-1,1]})
    \end{equation}
    and
    \begin{equation}
        \label{E:pf_good_inclusion2}
    E_1(\gamma) \cap \varphi_\gamma(G_\delta(z) \cap G_\delta(w)) \subset E_1(\Ec_{[-1,1]}).
    \end{equation}
To this end, consider the set of excursions $e \in \varphi_\gamma^{-1}(\Ec_{[-1,1]})$ with at least one endpoint in $(z_-,z_+)$. Define now $z_{\max,+}$ (resp. $z_{\max,-}$) to be the endpoint of such an excursion that lies on the counterclockwise arc $(z,w)$ (resp. $(w,z)$) and whose distance to $z$ is maximal. See Figure \ref{fig:zmax+} for an illustration.
Observe that
\begin{itemize}
    \item excursions $e \in \varphi_\gamma^{-1}(\Ec_{[-1,1]})$ with at least one endpoint in $[z_{\max,+},z_{\max,-}]$ belong to $\varphi_\gamma^{-1}(e_{\max})$;
    \item excursions $e \in \varphi_\gamma^{-1}(\Ec_{[-1,1]})$ with both endpoints in $(z_-,z_+)$ do not belong to $\varphi_\gamma^{-1}(e_{\max})$.
\end{itemize}
On the event $\varphi_\gamma(G_\delta(z))$, the excursions $e \in \varphi_\gamma^{-1}(\Ec_{[-1,1]})$ that intersect $B(z,\delta^{100})$ have both endpoints in $(z_-,z_+)$ and therefore do not belong to $\varphi_\gamma^{-1}(e_{\max})$. This shows \eqref{E:pf_good_inclusion1}.

On the event $\varphi_\gamma(G_\delta(z))$, the excursions with at least one endpoint in $(z_-,z_+)$ have diameter at most $\delta^{1/10}$. On this event, we thus have $|z_{\max,\pm} - z| \leq \delta^{1/10}$ and, because $|z-w|>\delta^{1/100}$, the arc $(w_-,w_+)$ is included in $[z_{\max,+},z_{\max,-}]$. In addition, on the event $\varphi_\gamma(G_\delta(w))$, the only excursions $e \in \varphi_\gamma^{-1}(\Ec_{[-1,1]})$ that visit $B(w,\delta^{100})$ have both endpoints in $(w_-,w_+)$. On $\varphi_\gamma(G_\delta(z) \cap G_\delta(w))$, these excursions must have both endpoints in $[z_{\max,+},z_{\max,-}]$ and therefore belong to $\varphi_\gamma^{-1}(e_{\max})$. This shows that, on $\varphi_\gamma(G_\delta(z) \cap G_\delta(w))$, the only excursion $e \in \Ec_{[-1,1]}$ that visits $\varphi_\gamma(B(w,\delta^{100}))$ is $e_{\max}$. Because on $E_1(\gamma)$ we have $\inte(\gamma) \cap \Hb_{\eta_2} \subset \varphi_\gamma(B(w,\delta^{100}))$, we obtain \eqref{E:pf_good_inclusion2}.

\medskip
\noindent\textbf{Change of measure.} In this step, we are going to change the law $\P_{\Hb,[-1,1]}^\wired (\cdot \vert X+iY)$ into a new law $\tilde \P_{\Hb,[-1,1]}^\wired (\cdot \vert X+iY)$ which is mutually absolutely continuous with respect to $\P_{\Hb,[-1,1]}^\wired(\cdot \vert X+iY)$, with a Radon--Nikodym derivative close to 1. As before, $X+iY$ will correspond to the point with maximal imaginary part. This change of measure is motivated by the following fact. Because of the interdependence of the endpoints $a(e), b(e), e \in \Ec_{[-1,1]}$, the maximal excursion $e_{\max}$ is \emph{not} independent of the other excursions $\{ e \in \Ec_{[-1,1]} \setminus \{e_{\max}\} \}$ under $\P_{\Hb,[-1,1]}^\wired(\cdot \vert \gamma)$. On the other hand, the excursion $e_{\max}$, or rather the top part $e_{\max}^2 \wedge e_{\max}^3$ of $e_{\max}$ (see below for precise definitions), will be independent of $\{ e \in \Ec_{[-1,1]} \setminus \{e_{\max}\} \}$ under the new law $\tilde \P_{\Hb,[-1,1]}^\wired(\cdot \vert \gamma)$.

    Let $U \in \partial \Hb_{\eta_2}$ and $V \in \partial \Hb_{\eta_2}$ be the first entrance and last exit of $e_{\max}$ in $\Hb_{\eta_2}$. We can decompose $e_{\max}$ into the concatenation $e_{\max} = e_{\max}^1 \wedge e_{\max}^2 \wedge e_{\max}^3 \wedge e_{\max}^4$ where $e_{\max}^1$  (resp. $e_{\max}^4$) is an excursion in $S_{\eta_2}$ from $a(e_{\max})$ to $U$ (resp. from $V$ to $b(e_{\max})$) and $e_{\max}^2$ (resp. $e_{\max}^3$) is an excursion in $S_{\max}$ from $U$ to $\varphi_\gamma(w)$ (resp. from $\varphi_\gamma(w)$ to $V$). Moreover, conditionally on the endpoints, the excursions $e_{\max}^i, i=1, \dots, 4$, are independent.

    Conditionally on $a(e_{\max})$, $b(e_{\max})$ and $\varphi_\gamma(w)$, the joint law of $(U,V)$ is given by
    \begin{align}
        & \P_{\Hb,[-1,1]}^\wired(U \in \d u, V \in \d v \vert a(e_{\max}), b(e_{\max}), \varphi_\gamma(w)) \\
        & \hspace{40pt} = H_{S_{\max}}(a(e_{\max}),\varphi_\gamma(w))^{-1} H_{S_{\max}}(b(e_{\max}),\varphi_\gamma(w))^{-1} \notag\\
        & \hspace{50pt} \times H_{S_{\eta_2}}(a(e_{\max}),u) H_{S_{\max}}(u,\varphi_\gamma(w)) H_{S_{\max}}(v,\varphi_\gamma(w)) H_{S_{\eta_2}}(b(e_{\max}),v).
        \notag 
    \end{align}
    These Poisson kernels turn out to be explicit; see \eqref{E:Poisson_strip1}-\eqref{E:Poisson_strip4}. It follows that the above display is further equal to
    \begin{align*}
    & \frac14 \frac{Y^2}{(Y -{\eta_2})^4} \Big(\cosh\Big(\pi \frac{a(e_{\max}) - X}{Y}\Big) + 1\Big) \Big(\cosh\Big(\pi \frac{b(e_{\max}) - X}{Y}\Big) + 1\Big) \\
    & \times \Big(\cosh\Big(\pi \frac{a(e_{\max}) - \Re (u)}{Y - {\eta_2}}\Big) + 1\Big)^{-1} \Big(\cosh\Big(\pi \frac{b(e_{\max}) - \Re (v)}{Y - {\eta_2}}\Big) + 1\Big)^{-1}\\
    & \times (\sin (\pi {\eta_2}/Y ))^2 \Big(\cosh\Big(\pi \frac{\Re(u) - X)}{Y}\Big) - \cos (\pi{\eta_2}/Y) \Big)^{-1} \Big(\cosh\Big(\pi \frac{\Re(v) - X)}{Y}\Big) - \cos (\pi{\eta_2}/Y)\Big)^{-1}.
    \end{align*}
Let $\wt U, \wt V \in \partial \Hb_{\eta_2}$ be i.i.d. with common distribution (the following expression is indeed a density, i.e. integrates to 1 by \eqref{E:Poisson_integral})
\[
\P(\wt U \in \d u \vert \varphi_\gamma(w)) = \frac12 \frac{\sin(\pi \delta/Y)}{Y-\delta} \Big(\cosh\Big(\pi \frac{\Re(u) - X)}{Y}\Big) - \cos (\pi\delta/Y)\Big)^{-1}.
\]
On the event $E_2(\gamma)$, $|U - X| \leq \sqrt{\delta}$ and $|V - X| \leq \sqrt{\delta}$. In particular, on this event,
\begin{equation}\label{E:pf_RN1}
    \frac{1}{1+\eta} \leq \frac{\P_{\Hb,[-1,1]}^\wired(U \in \d u, V \in \d v \vert \gamma)}{\Prob{\wt U \in \d u, \wt V \in \d v \vert \gamma}} \leq 1+\eta
\end{equation}
for some $\eta = \eta(\varphi_\gamma(w))$ that goes to zero as $\eta_2 \to 0$.

We can now define the aforementioned law $\tilde \P_{\Hb,[-1,1]}^\wired(\cdot \vert X+iY)$. Under this law, $\Pc$ is, as before, the concatenation of excursions in $\Hb$ attached to $[-1,1]$, except that the law of the highest excursion $e_{\max}$ is different.
It is the concatenation of four excursions $e_{\max}^i, i=1, \dots, 4,$ where $e_{\max}^1$  (resp. $e_{\max}^4$) is an excursion in $S_{\eta_2}$ from $a(e_{\max})$ to $\tilde U$ (resp. from $\tilde V$ to $b(e_{\max})$) and $e_{\max}^2$ (resp. $e_{\max}^3$) is an excursion in $S_{\max}$ from $\tilde U$ to $X+iY$ (resp. from $X+iY$ to $\tilde V$).
The only difference with $\P_{\Hb,[-1,1]}^\wired(\cdot \vert X+iY)$ is that the law of the first entrance and last exit of $e_{\max}$ in $\Hb_{\eta_2}$ are given by the law of $(\tilde U, \tilde V)$.

Let $\Fc$ be the sigma algebra generated by $\{e \in \Ec_{[-1,1]} \setminus \{e_{\max} \} \}$. The probability measure $\tilde \P_{\Hb,[-1,1]}^\wired$ has the desired property that, under this law,
$e_{\max}^2$ and $e_{\max}^3$ are independent of $\Fc$.

Before moving on, we rephrase slightly \eqref{E:pf_RN1} and claim that for any nonnegative measurable function $F$, and on the event $E_2(\gamma)$,
\begin{equation}
\label{E:pf_RN}
\frac{1}{1+\eta} \tilde \E^\wired_{\Hb,[-1,1]}[F(\Pc) \vert \gamma] \leq
\E^\wired_{\Hb,[-1,1]}[F(\Pc) \vert \gamma] \leq (1+\eta) \tilde \E^\wired_{\Hb,[-1,1]}[F(\Pc) \vert \gamma].
\end{equation}
Indeed, from \eqref{E:pf_RN1} and the definition of $\tilde \E^\wired_{\Hb,[-1,1]}[F(\Pc)]$, it follows directly that for any nonnegative measurable function $F$,
\[
\E^\wired_{\Hb,[-1,1]}[F(\Pc)] \leq (1+\eta) \tilde \E^\wired_{\Hb,[-1,1]}[F(\Pc)].
\]
In particular, for any nonnegative measurable functions $F$ and $G$,
\begin{align*}
\E^\wired_{\Hb,[-1,1]}[ G(\gamma) \E^\wired_{\Hb,[-1,1]}&[F(\Pc)\vert \gamma]]
= \E^\wired_{\Hb,[-1,1]}[ G(\gamma) F(\Pc)] \\
& \leq (1+\eta) \E^\wired_{\Hb,[-1,1]}[ G(\gamma) F(\Pc)]
= (1+\eta) \tilde \E^\wired_{\Hb,[-1,1]}[ G(\gamma) \tilde \E^\wired_{\Hb,[-1,1]}[F(\Pc)\vert \gamma]]
\end{align*}
which shows the second inequality in \eqref{E:pf_RN}. The first inequality is similar. Note that these inequalities will immediately transfer to inequalities concerning correlation functions.

\medskip
\noindent\textbf{Wrapping up.}
By \eqref{E:pf_good_inclusion1} and \eqref{E:pf_good_inclusion2}, on the event $E_1(\gamma)$, we have
\begin{align*}
    & \E_{\Hb,[-1,1]}^\wired[L_{\varphi_\gamma(z_\eps)}(\Pc) L_{\varphi_\gamma(w_\eps)}(\Pc) \mathbf{1}_{\varphi_\gamma(G_\delta(z) \cap G_\delta(w))} \vert \gamma] \\
    & \leq \E_{\Hb,[-1,1]}^\wired[L_{\varphi_\gamma(z_\eps)}(\Pc) L_{\varphi_\gamma(w_\eps)}(\Pc) \mathbf{1}_{E_1(\Ec_{[-1,1]}) \cap E_2(\Ec_{[-1,1]})} \vert \gamma].
\end{align*}
Moreover, on $E_1(\gamma) \cap E_1(\Ec_{[-1,1]})$ (resp. on $E_1(\gamma) \cap E_2(\Ec_{[-1,1]})$), $e_{\max}$ is the only excursion that contributes to the local time at $\varphi_\gamma(w_\eps)$ (resp. $e_{\max}$ does not contribute to the local time at $\varphi_\gamma(z_\eps)$). That is, on $E_1(\gamma)$,
\begin{align*}
    & \E_{\Hb,[-1,1]}^\wired[L_{\varphi_\gamma(z_\eps)}(\Pc) L_{\varphi_\gamma(w_\eps)}(\Pc) \mathbf{1}_{\varphi_\gamma(G_\delta(z) \cap G_\delta(w))} \vert \gamma] \\
    & \leq \E_{\Hb,[-1,1]}^\wired[L_{\varphi_\gamma(z_\eps)}(\Pc \setminus e_{\max}) L_{\varphi_\gamma(w_\eps)}(e_{\max}^2 \wedge e_{\max}^3) \mathbf{1}_{E_1(\Ec_{[-1,1]}) \cap E_2(\Ec_{[-1,1]})} \vert \gamma] \\
    & \leq \E_{\Hb,[-1,1]}^\wired[L_{\varphi_\gamma(z_\eps)}(\Pc \setminus e_{\max}) L_{\varphi_\gamma(w_\eps)}(e_{\max}^2 \wedge e_{\max}^3) \vert \gamma].
\end{align*}
By \eqref{E:pf_RN} and on $E_2(\gamma)$, the right hand side is at most
\begin{align*}
& (1+\eta) \tilde \E_{\Hb,[-1,1]}^\wired[L_{\varphi_{\gamma}(z_\eps)}(\Pc\setminus e_{\max}) L_{\varphi_{\gamma}(w_\eps)}(e_{\max}^2 \wedge e_{\max}^3) \vert \gamma] \\
& = (1+\eta) \tilde \E_{\Hb,[-1,1]}^\wired[L_{\varphi_{\gamma}(z_\eps)}(\Pc \setminus e_{\max}) \tilde \E_{\Hb,[-1,1]}^\wired[L_{\varphi_{\gamma}(w_\eps)}(e_{\max}^2 \wedge e_{\max}^3) \vert \Fc, \gamma] \vert \gamma].
\end{align*}
    We now focus on $\tilde \E_{\Hb,[-1,1]}^\wired[L_{\varphi_{\gamma}(w_\eps)}(e_{\max}^2 \wedge e_{\max}^3) \vert \Fc, \gamma]$. By independence of $e_{\max}^2 \wedge e_{\max}^3$ and $\Fc$ under $\tilde \P_{\Hb,[-1,1]}^\wired(\cdot \vert \gamma)$,
    \[
    \tilde \E_{\Hb,[-1,1]}^\wired[L_{\varphi_{\gamma}(w_\eps)}(e_{\max}^2 \wedge e_{\max}^3) \vert \Fc, \gamma]
    = \tilde \E_{\Hb,[-1,1]}^\wired[L_{\varphi_{\gamma}(w_\eps)}(e_{\max}^2 \wedge e_{\max}^3) \vert \gamma].
    \]
    Moreover, by \eqref{E:pf_RN} and on the event $E_2(\gamma)$,
    \begin{align*}
        & \tilde \E_{\Hb,[-1,1]}^\wired[L_{\varphi_{\gamma}(w_\eps)}(e_{\max}^2 \wedge e_{\max}^3) \vert \gamma]
        \leq (1+\eta) \E_{\Hb,[-1,1]}^\wired[L_{\varphi_{\gamma}(w_\eps)}(e_{\max}^2 \wedge e_{\max}^3) \vert \gamma] \\
        & \leq (1+\eta) \E_{\Hb,[-1,1]}^\wired[L_{\varphi_{\gamma}(w_\eps)}(\Pc) \vert \gamma] = (1+\eta) 5/\pi.
    \end{align*}
    Wrapping up, we have proved that, on $E_1(\gamma) \cap E_2(\gamma)$,
    \begin{align*}
        & \E_{\Hb,[-1,1]}^\wired[L_{\varphi_\gamma(z_\eps)}(\Pc) L_{\varphi_\gamma(w_\eps)}(\Pc) \mathbf{1}_{\varphi_\gamma(G_\delta(z) \cap G_\delta(w))} \vert \gamma]
        \leq (1+\eta)^2 \frac5\pi \E[L_{\varphi_{\gamma}(z_\eps)}(\Pc \setminus e_{\max}) \vert \gamma] \\
        & \leq (1+\eta)^2 \frac5\pi \E[L_{\varphi_{\gamma}(z_\eps)}(\Pc) \vert \gamma] = (1+\eta)^2 \Big(\frac5\pi \Big)^2.
    \end{align*}
    By \eqref{E:important} and since $\P(E_1(\gamma) \cap E_2(\gamma)) >0$, we have obtained that
    \[
    \Eint[L_{z_\eps}(\Pc)L_{w_\eps}(\Pc) \mathbf{1}_{G_\delta(z) \cap G_\delta(w)}] \leq (1+\eta)^2 \Big(\frac5\pi \Big)^2.
    \]
    Because $\eta \to 0$ as $\eps \to 0$ and then $\eta_2 \to 0$, this shows \eqref{E:pf_goal2} which concludes the proof.
\end{proof}
\end{proof}

\paragraph*{Acknowledgements}
We thank the anonymous referee for reading the paper carefully and for many useful suggestions.
During the process of writing this article, AJ was supported by Eccellenza grant 194648 of the Swiss National Science Foundation and was a member of NCCR SwissMAP. WQ is partially supported by a GRF grant from the Research Grants Council of the Hong Kong SAR (project CityU11305823).

\bibliographystyle{plain}
\small{\bibliography{bibliography}}

@article{Lawler2000TheDO,
  title={The dimension of the planar {B}rownian frontier is 4/3},
  author={Gregory F. Lawler and Oded Schramm and Wendelin Werner},
  journal={Mathematical Research Letters},
  year={2000},
  volume={8},
  pages={401-411}
}

@article {MR3101840,
    AUTHOR = {Schramm, Oded and Sheffield, Scott},
     TITLE = {A contour line of the continuum {G}aussian free field},
   JOURNAL = {Probab. Theory Related Fields},
  FJOURNAL = {Probability Theory and Related Fields},
    VOLUME = {157},
      YEAR = {2013},
    NUMBER = {1-2},
     PAGES = {47--80},
      ISSN = {0178-8051},
   MRCLASS = {60J67 (60G15)},
  MRNUMBER = {3101840},
MRREVIEWER = {Fredrik Johansson Viklund},
       DOI = {10.1007/s00440-012-0449-9},
       URL = {http://dx.doi.org/10.1007/s00440-012-0449-9},
}

@article {AidekonHuShi2018,
    AUTHOR = {A\"{\i}d\'{e}kon, Elie and Hu, Yueyun and Shi, Zhan},
     TITLE = {Points of infinite multiplicity of planar {B}rownian motion:
              measures and local times},
   JOURNAL = {Ann. Probab.},
  FJOURNAL = {The Annals of Probability},
    VOLUME = {48},
      YEAR = {2020},
    NUMBER = {4},
     PAGES = {1785--1825},
      ISSN = {0091-1798},
   MRCLASS = {60J65 (60J55)},
  MRNUMBER = {4124525},
       DOI = {10.1214/19-AOP1407},
       URL = {https://doi-org.ezp.lib.cam.ac.uk/10.1214/19-AOP1407},
}

@article {jegoBMC,
    AUTHOR = {Jego, Antoine},
     TITLE = {Planar {B}rownian motion and {G}aussian multiplicative chaos},
   JOURNAL = {Ann. Probab.},
  FJOURNAL = {The Annals of Probability},
    VOLUME = {48},
      YEAR = {2020},
    NUMBER = {4},
     PAGES = {1597--1643},
      ISSN = {0091-1798},
   MRCLASS = {60J65 (60J55)},
  MRNUMBER = {4124521},
       DOI = {10.1214/19-AOP1399},
       URL = {https://doi-org.ezp.lib.cam.ac.uk/10.1214/19-AOP1399},
}

@article{bass1994,
author = "Bass, Richard F. and Burdzy, Krzysztof and Khoshnevisan, Davar",
doi = "10.1214/aop/1176988722",
fjournal = "The Annals of Probability",
journal = "Ann. Probab.",
month = "04",
number = "2",
pages = "566--625",
publisher = "The Institute of Mathematical Statistics",
title = "Intersection Local Time for Points of Infinite Multiplicity",
url = "https://doi.org/10.1214/aop/1176988722",
volume = "22",
year = "1994"
}

@article {MR2045953,
    AUTHOR = {Lawler, Gregory F. and Werner, Wendelin},
     TITLE = {The {B}rownian loop soup},
   JOURNAL = {Probab. Theory Related Fields},
  FJOURNAL = {Probability Theory and Related Fields},
    VOLUME = {128},
      YEAR = {2004},
    NUMBER = {4},
     PAGES = {565--588},
      ISSN = {0178-8051},
   MRCLASS = {60J65 (81T40)},
  MRNUMBER = {2045953},
MRREVIEWER = {Andrea Posilicano},
       DOI = {10.1007/s00440-003-0319-6},
       URL = {https://doi.org/10.1007/s00440-003-0319-6},
}

@book {MR2129588,
    AUTHOR = {Lawler, Gregory F.},
     TITLE = {Conformally invariant processes in the plane},
    SERIES = {Mathematical Surveys and Monographs},
    VOLUME = {114},
 PUBLISHER = {American Mathematical Society},
   ADDRESS = {Providence, RI},
      YEAR = {2005},
     PAGES = {xii+242},
      ISBN = {0-8218-3677-3},
   MRCLASS = {60-02 (30-02 30C35 31A15 60J65 81T40 82B27)},
  MRNUMBER = {2129588 (2006i:60003)},
MRREVIEWER = {Zhen-Qing Chen},
}

@article{QianWerner19Clusters,
	title={Decomposition of {B}rownian loop-soup clusters},
	author={Qian, Wei and Werner, Wendelin},
	journal={J. Eur. Math. Soc.},
	volume={21},
	number={10},
	pages={3225--3253},
	year={2019}
}

@book{morters2010brownian,
  title={Brownian motion},
  author={M{\"o}rters, Peter and Peres, Yuval},
  volume={30},
  year={2010},
  publisher={Cambridge University Press}
}

@article{MR1992830,
  title={Conformal restriction: the chordal case},
  author={Lawler, Gregory and Schramm, Oded and Werner, Wendelin},
  journal={J. Amer. Math. Soc.},
  fjournal={Journal of the American Mathematical Society},
  volume={16},
  number={4},
  pages={917--955},
  year={2003}
}

@article{JLQ23b,
Author = {Antoine Jego and Titus Lupu and Wei Qian},
Title = {{Conformally invariant fields out of Brownian loop soups}},
journal = {arXiv:2307.10740},
Year = {2023}
}

@article{MR3901648,
author = {Wei Qian},
title = {{Conditioning a Brownian loop-soup cluster on a portion of its boundary}},
volume = {55},
journal = {Annales de l'Institut Henri Poincaré, Probabilités et Statistiques},
number = {1},
publisher = {Institut Henri Poincaré},
pages = {314 -- 340},
year = {2019},
doi = {10.1214/18-AIHP883},
URL = {https://doi.org/10.1214/18-AIHP883}
}

@article {MR2217292,
    AUTHOR = {Garban, Christophe and Trujillo Ferreras, Jos\'{e} A.},
     TITLE = {The expected area of the filled planar {B}rownian loop is
              {$\pi/5$}},
   JOURNAL = {Comm. Math. Phys.},
  FJOURNAL = {Communications in Mathematical Physics},
    VOLUME = {264},
      YEAR = {2006},
    NUMBER = {3},
     PAGES = {797--810},
      ISSN = {0010-3616,1432-0916},
   MRCLASS = {82B41 (60J65)},
  MRNUMBER = {2217292},
MRREVIEWER = {Robert\ Otto\ Bauer},
       DOI = {10.1007/s00220-006-1555-2},
       URL = {https://doi.org/10.1007/s00220-006-1555-2},
}

@article{widder1961functions,
  title={Functions harmonic in a strip},
  author={Widder, David V},
  journal={Proceedings of the American Mathematical Society},
  volume={12},
  number={1},
  pages={67--72},
  year={1961}
}

@article {MR2350053,
    AUTHOR = {Werner, Wendelin},
     TITLE = {The conformally invariant measure on self-avoiding loops},
   JOURNAL = {J. Amer. Math. Soc.},
  FJOURNAL = {Journal of the American Mathematical Society},
    VOLUME = {21},
      YEAR = {2008},
    NUMBER = {1},
     PAGES = {137--169},
      ISSN = {0894-0347},
   MRCLASS = {60D05 (30C35 60J65)},
  MRNUMBER = {2350053 (2009d:60028)},
MRREVIEWER = {Robert Otto Bauer},
       DOI = {10.1090/S0894-0347-07-00557-7},
       URL = {http://dx.doi.org/10.1090/S0894-0347-07-00557-7},
}

@article {MR2153402,
    AUTHOR = {Rohde, Steffen and Schramm, Oded},
     TITLE = {Basic properties of {SLE}},
   JOURNAL = {Ann. of Math. (2)},
  FJOURNAL = {Annals of Mathematics. Second Series},
    VOLUME = {161},
      YEAR = {2005},
    NUMBER = {2},
     PAGES = {883--924},
      ISSN = {0003-486X,1939-8980},
   MRCLASS = {60K35 (28A80 60J55 60J65)},
  MRNUMBER = {2153402},
MRREVIEWER = {Olivier\ Raimond},
       DOI = {10.4007/annals.2005.161.883},
       URL = {https://doi.org/10.4007/annals.2005.161.883},
}

@article {MR3786302,
    AUTHOR = {Gwynne, Ewain and Miller, Jason and Sun, Xin},
     TITLE = {Almost sure multifractal spectrum of {S}chramm-{L}oewner
              evolution},
   JOURNAL = {Duke Math. J.},
  FJOURNAL = {Duke Mathematical Journal},
    VOLUME = {167},
      YEAR = {2018},
    NUMBER = {6},
     PAGES = {1099--1237},
      ISSN = {0012-7094,1547-7398},
   MRCLASS = {60J67 (60G17)},
  MRNUMBER = {3786302},
       DOI = {10.1215/00127094-2017-0049},
       URL = {https://doi.org/10.1215/00127094-2017-0049},
}

@book {MR0665254,
    AUTHOR = {Mandelbrot, Benoit B.},
     TITLE = {The fractal geometry of nature},
 PUBLISHER = {W. H. Freeman and Co., San Francisco, CA},
      YEAR = {1982},
     PAGES = {v+460},
      ISBN = {0-7167-1186-9},
   MRCLASS = {00A69 (51-01 54F45 85A35 86A99)},
  MRNUMBER = {665254},
MRREVIEWER = {S.\ Dubuc},
}

@article {MR1879851,
    AUTHOR = {Lawler, Gregory F. and Schramm, Oded and Werner, Wendelin},
     TITLE = {Values of {B}rownian intersection exponents. {II}. {P}lane exponents},
   JOURNAL = {Acta Math.},
  FJOURNAL = {Acta Mathematica},
    VOLUME = {187},
      YEAR = {2001},
    NUMBER = {2},
     PAGES = {275--308},
      ISSN = {0001-5962,1871-2509},
   MRCLASS = {60J65 (30C35 82B41)},
  MRNUMBER = {1879851},
MRREVIEWER = {Christophe\ Giraud},
       DOI = {10.1007/BF02392619},
       URL = {https://doi.org/10.1007/BF02392619},
}

@article {MR1961197,
    AUTHOR = {Lawler, Gregory F. and Schramm, Oded and Werner, Wendelin},
     TITLE = {Analyticity of intersection exponents for planar {B}rownian
              motion},
   JOURNAL = {Acta Math.},
  FJOURNAL = {Acta Mathematica},
    VOLUME = {189},
      YEAR = {2002},
    NUMBER = {2},
     PAGES = {179--201},
      ISSN = {0001-5962,1871-2509},
   MRCLASS = {60J65 (28A80 47A10 60G17 60K35)},
  MRNUMBER = {1961197},
MRREVIEWER = {Ren\'{e}\ L.\ Schilling},
       DOI = {10.1007/BF02392842},
       URL = {https://doi.org/10.1007/BF02392842},
}

@article {MS,
    AUTHOR = {Miller, Jason  and Sheffield, Scott},
     TITLE = {{CLE}(4) and the {G}aussian {F}ree {F}ield},
   JOURNAL = {In preparation},
}

@book {MR1121940,
    AUTHOR = {Karatzas, Ioannis and Shreve, Steven E.},
     TITLE = {Brownian motion and stochastic calculus},
    SERIES = {Graduate Texts in Mathematics},
    VOLUME = {113},
   EDITION = {Second},
 PUBLISHER = {Springer-Verlag, New York},
      YEAR = {1991},
     PAGES = {xxiv+470},
      ISBN = {0-387-97655-8},
   MRCLASS = {60J65 (35K99 35R60 60G44 60H10 60J60)},
  MRNUMBER = {1121940},
       DOI = {10.1007/978-1-4612-0949-2},
       URL = {https://doi.org/10.1007/978-1-4612-0949-2},
}

@article{qian2018,
author = "Qian, Wei and Werner, Wendelin",
doi = "10.1214/18-EJP258",
fjournal = "Electronic Journal of Probability",
journal = "Electron. J. Probab.",
pages = "23 pp.",
pno = "128",
publisher = "The Institute of Mathematical Statistics and the Bernoulli Society",
title = "{The law of a point process of Brownian excursions in a domain is determined by the law of its trace}",
url = "https://doi.org/10.1214/18-EJP258",
volume = "23",
year = "2018"
}

@article{Taylor_1964,
title={{The exact Hausdorff measure of the sample path for planar Brownian motion}},
volume={60},
DOI={10.1017/S0305004100037713},
number={2},
journal={Mathematical Proceedings of the Cambridge Philosophical Society},
author={Taylor, S. J.},
year={1964},
pages={253–258}}

@article{zbMATH02055258,
 author = {Vir{\'a}g, B{\'a}lint},
 title = {Brownian beads},
 fjournal = {Probability Theory and Related Fields},
 journal = {Probab. Theory Relat. Fields},
 issn = {0178-8051},
 volume = {127},
 number = {3},
 pages = {367--387},
 year = {2003},
 language = {English},
 doi = {10.1007/s00440-003-0289-8},
 zbMATH = {2055258},
 Zbl = {1035.60085}
}

\end{document}